\documentclass[11pt]{amsart}
\usepackage{setspace}
\usepackage{xcolor}
\usepackage{makecell}
\usepackage{amssymb,amscd,amsthm,verbatim,amsmath,color,fancyhdr,mathrsfs,amsfonts,amssymb,commath, graphicx,bbm}
\usepackage[bookmarks=false]{hyperref}
\usepackage[english]{babel}
\usepackage{parskip}
\usepackage{enumitem}

\usepackage[letterpaper,top=2cm,bottom=2cm,left=3cm,right=3cm,marginparwidth=1.75cm]{geometry}

\DeclareMathOperator{\Id}{\mathrm{I}_{d}}

\newcommand{\Was}[1]{\mathbb{W}_{#1}}

\newcommand{\Cov}{\mathrm{Cov}}

\DeclareMathOperator*{\argmin}{argmin}

\newcommand{\cP}{\mathcal{P}}

\newcommand{\vol}{\mathrm{vol}}
\newcommand{\diffref}{\mathbf{m}}
\newcommand{\ricci}{\mathrm{Ric}}

\newcommand{\Tr}{\text{Tr}}

\newcommand{\Schro}{\text{Schr\"{o}dinger}}

\newcommand{\Exp}[1]{\mathrm{E}_{#1}}

\newtheorem{remark}{Remark}

\newtheorem{definition}{Definition}

\newtheorem{theorem}{Theorem}
\newtheorem{assumption}{Assumption}
\newtheorem{proposition}{Proposition}

\newtheorem{mainresult}{Main Result}

\newcommand{\commentout}[1]{}

\title[Noising-Denoising by Large Temperature Schr\"{o}dinger Bridges]{Noising-Denoising by Large Temperature Schr\"{o}dinger Bridges}
\author{Garrett Mulcahy}
\address{Garrett Mulcahy\\ Department of Mathematics \\ University of Washington\\ Seattle, WA 98195, USA \\ {Email: gmulcahy@uw.edu}}

	

\thanks{This research is partially supported by the following grants: NSF DMS-2502281, 2133244, 2052239 and PIMS PRN-01 Kantorovich Initiative. Many thanks to Soumik Pal for helpful discussion and comments.}

\begin{document}

\begin{abstract}
In this note, we establish a connection between denoising diffusion models and large temperature dynamic Schr\"{o}dinger bridges computed with respect to certain reference processes. In particular, we show that in the large temperature regime (i.e.\ $T \to +\infty$), the Schr\"{o}dinger bridge behaves like a forward time diffusion process over the time interval $[0,T/2]$, and approximately like the time reversal of a diffusion process over $[T/2,T]$. This construction provides a process-level characterization of the gradient flow approximation developed by \cite{clerc23-longtime}.  
\end{abstract}
\maketitle

\section{Introduction}
Noising-denoising diffusion models are a dominant paradigm in generative modeling. Introduced and first investigated in the works \cite{sohl-dickstein15,song2019Generative,song2020score,ho2020denoising}, the idea is the following. First, one begins with samples from a distribution and then runs, given this initial data, a diffusion with a known mixing rate. This is called the \textbf{noising} step. The most popular choice of diffusion is the Ornstein-Uhlenbeck (OU) process, 
\begin{align}\label{eq:OU}\tag{OU}
    dX_t &= -\frac{1}{2}X_t dt + dB_t,
\end{align}
which has stationary distribution $N(0,\Id)$, where $\Id$ denotes the $d \times d$ identity matrix. At training time, one fixes a large time horizon $T > 0$ and learns via a neural network an approximation to the score function of each marginal time point. At inference time, i.e.\ when one wishes to generate a sample, one performs a time reversal and transforms samples from the stationary distribution into samples from the initial distribution. This step is called \textbf{denoising}.

In this note, we argue that the \textbf{large temperature $\Schro$ bridge}, a solution to a certain entropy minimization problem (see the definition in \eqref{def:sb-dyn-ou}), behaves like a noising-denoising diffusion model up to a reasonable precision. As we elaborate in Section \ref{subsec:sb-versus-denoising}, the large temperature $\Schro$ bridge may present both theoretical and computational advantages compared to the dominant noising-denoising diffusion models. For the sake of the Introduction, we will introduce our main ideas on $\mathbb{R}^{d}$ equipped with the OU process defined in \eqref{eq:OU}. However, our results also extend to the compact manifold setting. Note that denoising diffusion models have been considered in the manifold setting in \cite{debortoli2022riemannian}. We will fully introduce those ideas and notation later when they become relevant in Sections \ref{sec:preliminaries} and \ref{sec:manifold}.

Fix $T > 0$. We call this parameter temperature and will take it to be large in this note. Let $R^{T} = \mathrm{Law}(X_{t}, t \in [0,T])$, where $(X_t, t \geq 0)$ satisfies \eqref{eq:OU} with initial distribution $X_0 \sim N(0,\Id)$, i.e.\ with initial distribution given by its stationary distribution. We will set $\diffref$ to be $N(0,\Id)$, and we will call it the \textbf{reference measure} and  \eqref{eq:OU} the \textbf{reference process}. Let $H(\cdot|\cdot)$ denote the relative entropy (i.e.\ Kullback-Leibler divergence) between two probability measures. Given two probability measures $\mu,\nu \in \cP(\mathbb{R}^{d})$ with finite second moments, the (dynamic) $T$-\textbf{$\Schro$ bridge} from $\mu$ to $\nu$ with respect to the \eqref{eq:OU} reference process is the solution to the following optimization problem,
\begin{align}\label{def:sb-dyn-ou}
    P^{T} := \argmin \limits_{P \in \Pi(\mu,\nu)} H(P|R^{T}),
\end{align}
where $\Pi(\mu,\nu)$ is the set of all probability measures on the space of continuous paths (continuous functions from $[0,T]$ to $\mathbb{R}^{d}$), denoted $C([0,T];\mathbb{R}^{d})$, such that the time zero marginal is $\mu$ and the time $T$ marginal is $\nu$. We denote the coordinate process by $(x_t, t \in [0,T])$. 

\textbf{Our construction.} We construct an explicit approximation for $P^{T}$ based on a noising-denoising framework. In other words, we construct a path measure in $\Pi(\mu,\nu)$ that behaves like a forward time diffusion on $[0,T/2]$, and then approximately like the time reversal of a diffusion on $[T/2,T]$. 

The components are as follows. Let $(r_{t}(\cdot,\cdot), t \geq 0)$ denote the transition densities with respect to the stationary measure $\diffref(dy)$ of the reference dynamics. We denote the corresponding semigroup by $(R_t,t \geq 0)$ (precise definitions are given in Section \ref{sec:preliminaries}). As the reference dynamics are currently given by \eqref{eq:OU},
\begin{align}\label{eq:ou-trans-dens}
    r_{t}(x,y) &= \frac{1}{(1-e^{-t})^{d/2}}\exp\left(-\frac{\norm{x}^2-2e^{t/2}\langle x,y\rangle + \norm{y}^2}{2(e^{t}-1)}\right).
\end{align}
Next, define the Radon-Nikodym derivatives of the marginals with respect to $\diffref$ as
\begin{align}
    \rho := d\mu/d\diffref, \quad \sigma := d\nu/d\diffref.
\end{align}
First, start \eqref{eq:OU} with $X_0 \sim \mu$ and run it for time $T/2$. This is the \textbf{``noising step''}. The joint distribution at times $0$ and $T/2$ is denoted $\alpha_{T/2}$ and is given by
\begin{align*}
    \frac{d\alpha_{T/2}}{d(\diffref \otimes \diffref)}(x_0,x_{T/2}) &= \rho(x_0) r_{T/2}(x_0,x_{T/2}).
\end{align*}
Now, do the same procedure but with initial distribution $\nu$. Call this joint distribution $\beta_{T/2}$. We now compute the conditional distribution of the position at time $0$ given the position at time $T/2$. This is the \textbf{``denoising step''}, and the density is given by
\begin{align*}
    \frac{d\beta_{T/2}(\cdot|y_{T/2})}{d\diffref}(y_0) &= \frac{\sigma(y_0) r_{T/2}(y_0,y_{T/2})}{\int_{\mathbb{R}^{d}} \sigma(y_0) r_{T/2}(y_0,y_{T/2}) \diffref(dy_0)} = \frac{\sigma(y_0)r_{T/2}(y_0,y_{T/2})}{R_{T/2}\sigma (y_{T/2})}. 
\end{align*}
Lastly, let $\nabla \varphi_{T/2}$ denote the quadratic cost optimal transport map from $R_{T/2}^*\mu$ to $R_{T/2}^*\nu$, defined later in Section \ref{subsec:ot-prelim}. This map is guaranteed to exist by \cite{THEbrenier}.

The construction is then the following: start \eqref{eq:OU} with $X_0 \sim \mu$ and run for time $T/2$ (noising). Then, pushforward the $T/2$ coordinate by $\nabla \varphi_{T/2}$ and run the time reversal of \eqref{eq:OU} started from $\nu$ (denoising). In other words, we have a triplet $(X_0,X_{T/2},X_{T})$ such that the joint distribution has density with respect to $\diffref^{\otimes 3}$ equal to
\begin{align}\label{eq:triplet-ou}
    \rho(x_0) r_{T/2}(x_0,x_{T/2}) \cdot \frac{\sigma(x_T)r_{T/2}(x_T,\nabla \varphi_{T/2}(x_{T/2}))}{R_{T/2}\sigma (\nabla \varphi_{T/2}(x_{T/2}))}
\end{align}
From this joint density, we define a path measure $Q^T$ with Radon-Nikodym derivative
\begin{align}\label{defn:diffusion-approx}
    \frac{dQ^{T}}{dR^{T}}(x) &:= \frac{\rho(x_0)\sigma(x_T)}{R_{T/2}\sigma(\nabla \varphi_{T/2}(x_{T/2}))} \cdot \frac{r_{T/2}(\nabla \varphi_{T/2}(x_{T/2}),x_{T})}{r_{T/2}(x_{T/2},x_{T})}
\end{align}
By construction, observe that $Q^{T} \in \Pi(\mu,\nu)$ and thus is an admissible candidate for the $\Schro$ bridge. Importantly, observe that $Q^{T}$ behaves like a ``noising'' process over the time interval $[0,T/2]$, and approximately like a ``denoising'' process over $[T/2,T]$. Note that over $[T/2,T]$ the process is not exactly the time reversal of the gradient flow started from $\nu$ due to the pushforward correction. A pictorial representation of $Q^{T}$ is given in Figure \ref{fig:large-time-schematic}. 

Note that when $\mu =\nu$ the map $\nabla \varphi_{T/2}$ is the identity, so the construction becomes much simpler. In fact, in this case the process over $[T/2,T]$ is exactly the time reversal of the gradient flow started from $\nu$. In general though, for large $T$ the probability measures $R_{T/2}^*\mu$ and $R_{T/2}^*\nu$ are both very close to the stationary measure $\diffref$ (see discussion of the turnpike property in Section \ref{subsec:related-works}). Thus, intuitively $\nabla \varphi_{T/2}$ is very close to the identity map for large values of $T$, i.e.\ there is only a slight modification at time $T/2$. 

\begin{figure}
    \centering
    \includegraphics[width=1.0\linewidth]{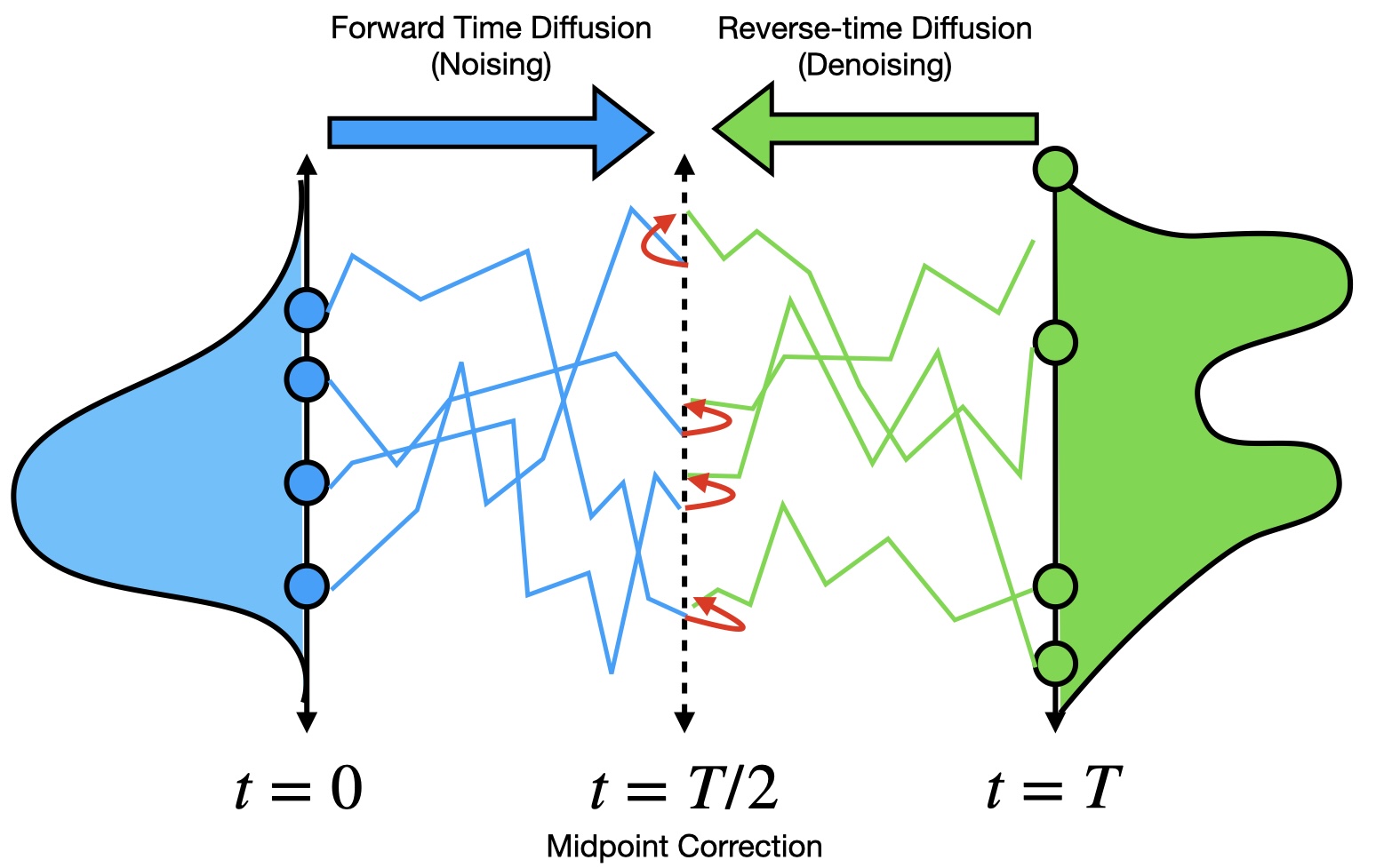}
    \caption{Sketch of the two step Markov chain $(X_0,X_{T/2},X_{T})$ from which $Q^{T}$ is derived in \eqref{defn:diffusion-approx}.}
    \label{fig:large-time-schematic}
\end{figure}

We summarize our main results as follows. 
\begin{mainresult}[Informal Statement of Theorems \ref{thm:euclidean-case} and \ref{thm:compact-manifold}]\label{mainresult}
    Let $P^{T}$ and $Q^{T}$ be as defined in \eqref{def:sb-dyn-ou} and \eqref{defn:diffusion-approx}. Under assumptions on $\mu$ and $\nu$ stated in Assumption \ref{assumption},
    \begin{align}\label{eq:sym-rel-informal}
        \limsup\limits_{T \to +\infty} \exp(T/2)\left(H(Q^{T}|P^{T})+H(P^{T}|Q^{T})\right) <+\infty. 
    \end{align}
    An analogous result also holds on compact manifolds $(M,g)$. In this case the exponential rate is governed by the spectral decomposition of the reference process.
\end{mainresult}
We also establish that this convergence rate cannot in general be improved, see Theorem \ref{thm:compact-manifold-same}. Note that by the information processing inequality, Main Result \ref{mainresult} provides a decay rate for the symmetric relative entropy of the joint distribution of any finite collection of time marginals of $Q^{T}$ and $P^{T}$. In particular, integrating \eqref{eq:triplet-ou} with respect to $\diffref(dx_{T/2})$ results in an explicit coupling of $\mu$ and $\nu$ that approximates the static $\Schro$ bridge (i.e.\ $(x_0,x_{T})_{\#}P^{T}$) with a decay rate that is no worse than that given in Main Result \ref{mainresult}. 

Lastly, we remark that Main Result \ref{mainresult} hints at the following large time structural result for $\Schro$ bridges. It says that the $\Schro$ bridge from $\mu$ to $\nu$ for large $T$ can be approximately obtained by combining the $\Schro$ bridges from $\mu$ to $\diffref$ and $\diffref$ to $\nu$, both computed at temperature $T/2$. In other words, at large temperature a semigroup-like property emerges in which $\Schro$ bridges with one marginal equal to the reference measure $\diffref$ play a crucial role. This is a large temperature result complementary to one developed in the small temperature regime established in \cite[Theorem 3]{AHMP25} and \cite[Theorem 2]{MP25}. For a rigorous statement and proof, see Theorem \ref{thm:semigrp-large-temp}.

\subsection{$\Schro$ Bridge and Denoising Diffusion Models}\label{subsec:sb-versus-denoising} 
The $\Schro$ bridge has been employed in a generative modeling context by works such as \cite{doucet2021diffusion,wang-dl-sb-21,tong-sbsim-24,shi2023diffusion,peluchetti-23,doucet-unpaired-24}. We note that \cite{zhang-sb-ouformula-2026,bunne-sbformula-26} specifically consider the $\Schro$ bridge with \eqref{eq:OU} reference in their analyses. In general, the $\Schro$ bridge possesses a body of theoretical and practical results that may make it a more desirable alternative to the more popular noising-denoising diffusion models in certain contexts. Many of these points are raised in \cite{doucet2021diffusion}. For instance, in the time-reversal denoising step, there is a truncation error arising from the fact that one begins with a sample from the stationary distribution, even though the stationary distribution is never attained in finite time. This issue is not present with the $\Schro$ bridge as the noising and denoising steps are performed at the same time. 

We highlight this comparison specifically in the same marginal case. As mentioned earlier, when $\mu = \nu$ the construction $Q^{T}$ in \eqref{defn:diffusion-approx} can be most faithfully interpreted as an approximate noising-denoising process as $\nabla \varphi_{T/2} = \Id$. In this case, $P^{T}$ and $Q^{T}$ are similar to so-called mirror interpolations as defined in \cite[Section 3.3]{ABV25}. This allows us to use algorithms for computing $P^{T}$ to obtain a stochastic interpolation that approximately operates as a noising-denoising model at large temperatures.

To compute $\Schro$ bridges, there are two algorithms with well-studied convergence theory: the Sinkhorn algorithm (Iterative Proportional Fitting Procedure) \cite{cuturi2013sinkhorn,sinkhorn-OG}, and the  Iterative Markov Fitting (IMF) algorithm \cite{shi2023diffusion,peluchetti-23}. Both algorithms are known to produce exponential convergence (with respect to the number of iterations) under assumptions related to log-concavity of the marginals, especially under large regularization (i.e.\ large $T$). For the Sinkhorn algorithm see \cite[Theorem 1.3, Proposition 1.11]{CCGT25} and example gallery therein; for IMF, see \cite[Theorems 1 and 2]{SCD-dsb-25} (although these results only apply to the Euclidean setting). There is a vast literature establishing convergence rates of the Sinkhorn algorithm in more general settings, see the recent works \cite{gn-sinkhorn-rate26,chizat-sinkconv26,cdg-sinkhkorn-26,eck-conv25} and references therein.  

At the same time, these theoretical convergence results do not guarantee strong performance in practical settings of interest. For instance, the implementation of IMF requires fitting a neural network to compute the Markov projection. Similarly for score-based denoising diffusion models, the ability of the neural network to efficiently estimate the score function is key to guaranteeing strong model performance \cite{chen2023sampling,chen2023improved,lee2023convergence,conforti2025kl}. We hope that the theoretical results of this note will encourage more investigation into the comparisons and tradeoffs of using \textbf{large} temperature $\Schro$ bridges as opposed to denoising diffusion models in generative modeling applications.

\subsection{Related Works}\label{subsec:related-works}
As $T \to +\infty$, it is known in various senses that the joint distribution at time $0$ and time $T$ of $P^{T}$ converges to the independent coupling $\mu \otimes \nu$. Additionally, there is interest in understanding the convergence of time points along the \textbf{entropic interpolation}, i.e.\ the curve of measures traced out by the time marginals of $P^{T}$, to the stationary measure $\diffref$. Set $\mu_{t}^{T} := (x_t)_{\#}P^{T}$, then we write $(\mu_{t}^{T},t \in [0,T])$ to denote the entropic interpolation. More precise descriptions of the large time limiting behavior of the $\Schro$ bridge have been given in works such as \cite{conforti-second-order-sb19,conforti21deriv,clerc23-longtime,tamanini22-costa}. 

\textbf{Curvature Dimension Condition.} An important point of contrast between the results of this note and previous works is the assumption that the reference process satisfies a curvature dimension condition. Using the manifold notions introduced in Section \ref{sec:preliminaries}, the definition is as follows. 
\begin{definition}[Curvature Dimension Condition]
    For $\kappa \in \mathbb{R}$, the triplet $(M,g,e^{-U}\vol)$ satisfies
$\text{CD}(\kappa,\infty)$ if the following inequality holds for all $x \in M$ 
\begin{align}\label{eq:cd-defn}\tag{CD($\kappa,\infty$)}
    \ricci + \mathrm{Hess}(U) \geq \kappa g. 
\end{align}
\end{definition}
Here, $\ricci_{x}$ denotes the Ricci curvature tensor at $x$. By associating the diffusion process defined later in \eqref{eq:reference-process-intro} with its stationary measure, we will also say that the process \eqref{eq:reference-process-intro} satisfies $\text{CD}(\kappa,\infty)$. Crucially, the standard Gaussian on $\mathbb{R}^{d}$, i.e.\ the stationary measure of \eqref{eq:OU}, satisfies $\text{CD}(1,\infty)$. This is an important condition that establishes many functional inequalities related to the convergence of the associated Markov process to its equilibrium measure, see \cite{bgl-markov} (although pay attention to differing constant conventions). 

While we will use the fact that $\mathbb{R}^{d}$ equipped with \eqref{eq:OU} satisfies $\mathrm{CD}(1,\infty)$, our proofs in the compact manifold setting will entirely avoid this assumption in favor of spectral arguments, see Section \ref{subsec:spectral-analysis} for more details. To compare previous results with ours, the key bridge between spectral methods and the curvature dimension condition is the following. Let $\lambda_1$ denote the so-called spectral gap of the spectral decomposition of $L^2(\diffref)$, i.e.\ the first nonzero eigenvalue of the (negative) generator of the reference process (defined later in \eqref{defn:generator-of-diff}). If the reference process also satisfies $\mathrm{CD}(\kappa,\infty)$ for some $\kappa \in \mathbb{R}$, then 
\begin{align}\label{remark:cd-lambda}
    \text{it holds that $\lambda_1 \geq \kappa/2$.}
\end{align}
As \eqref{eq:OU} satisfies $\text{CD}(1,\infty)$, in this case the inequality in \eqref{remark:cd-lambda} is an equality. In general though this inequality can be strict. For instance, take the sphere $\mathbb{S}^{d}$ equipped with Brownian motion reference process (i.e.\ its generator is $\frac{1}{2}\Delta_{\mathbb{S}^{d}}$). In this case, the optimal value for $\kappa$ is $d-1$ as $\ricci = (d-1)g$, whereas $\lambda_1 = d/2$. 

We stress that in the general context of Markov processes, especially on non-compact spaces and outside the $\mathrm{CD}(\kappa,\infty)$ with $\kappa > 0$ setting, the spectrum of the generator is in general \textbf{not} discrete \cite{milman18-spectral}. However, we work with the spectral setting specifically because it is sufficiently general to capture a connection with noising-denoising diffusion models. 

\textbf{Entropic Cost and Interpolation.} Under a general curvature dimension assumption and minor additional assumptions, \cite[Theorem 1.4]{conforti21deriv} establishes the following inequality for all $T > 0$ under $\mathrm{CD}(\kappa,\infty)$ with $\kappa > 0$ and $\mu,\nu$ having bounded compactly supported densities with respect to $\diffref$,
\begin{align}\label{eq:conv-rate-ct}
    \abs{H(P^{T}|R^{T})-\left(H(\mu|\diffref)+H(\nu|\diffref)\right)} \leq \frac{2}{\exp(\kappa T/2)-1}(H(\mu|\diffref)+H(\nu|\diffref)).
\end{align}
Moreover, the same theorem establishes that this exponential rate of $\kappa/2$ in general cannot be improved. We also use many large temperature estimates and results from \cite{conforti21deriv} that we summarize later in Proposition \ref{prop:schro-limit-facts}. Additionally, \cite[Theorem 1.4]{conforti-second-order-sb19} and \cite[Equation (2.20)]{conforti21deriv} establish convergence rates in relative entropy of each time point $\mu_{t}^{T}$ along the entropic interpolation to the reference measure $\diffref$ as $T \to +\infty$. To be precise, it is shown under the same assumptions as \eqref{eq:conv-rate-ct} that for all $T >0$ and $t \in [0,T]$
\begin{align}\label{eq:rel-ent-equilibrium-conv}
    H(\mu_t^{T}|\diffref) &\leq \frac{1-\exp(-\kappa (T-t))}{1-\exp(-\kappa T)}H(\mu|\diffref) + \frac{1-\exp(-\kappa t)}{1-\exp(-\kappa T)}H(\nu|\diffref) \\
    &- \frac{\cosh(\kappa T/2)-\cosh(\kappa (t-T/2))}{\sinh(\kappa T/2)}H(P^{T}|R^{T}). 
\end{align}
Indeed, results similar to Theorems \ref{thm:euclidean-case} and \ref{thm:compact-manifold} can be established on the basis of the inequalities developed in \eqref{eq:rel-ent-equilibrium-conv} in terms of $\kappa$ instead of $\lambda_1$. However, we establish our results by relying instead on the spectral decomposition of $L^2(\diffref)$ available in our setting. 

\textbf{Large Temperature $\Schro$ Bridge and Gradient Flows.} This work is heavily inspired by the large temperature $\Schro$ bridge analysis done in \cite{clerc23-longtime}. The main result of \cite{clerc23-longtime} is a bound on the $\Was{2}$ distance between each time point $\mu_{t}^{T}$ along the entropic interpolation and the distribution of the corresponding time point of the reference dynamics with initial measure $\mu$, i.e.\ $R_t^*\mu$, in both the $\mathrm{CD}(\kappa,\infty)$ and $\mathrm{CD}(0,n)$ settings. The curve of measures $(R_t^*\mu, t \geq 0)$ is the Wasserstein gradient flow (curve of steepest descent) of the functional $H(\cdot|\diffref)$ started from $\mu$, see \cite{ambrosio2005gradient} for more on this characterization. In our time scaling and $\mathrm{CD}(\kappa,\infty)$ setting, when $\mu,\nu$ have compactly supported, smooth densities with respect to $\diffref$, \cite[Theorem 4.2]{clerc23-longtime} gives for all $t \in [0,T)$ 
\begin{align}\label{eq:clerc-bdd}
     \Was{2}(\mu_t^{T}, R_t^* \mu) &\leq C\frac{t\exp(-\kappa T/2)}{\sqrt{\exp(-\kappa t)-\exp(-\kappa T)}}.
\end{align}
Observe that this bound explodes as $t \to T$, so this bound is best understood as fixing $t \in (0,T)$ and then sending $T \to +\infty$.

By comparison, our Main Result \ref{mainresult} provides a process-level approximation to the large temperature $\Schro$ bridge at each temperature $T > 0$. By the information processing inequality, Main Result \ref{mainresult} implies a uniform bound over all $t \in [0,T]$ of the relative entropy of $\mu_{t}^{T}$ with respect to the time $t$ marginal of $Q^T$ defined in \eqref{defn:diffusion-approx}. Indeed, the motivation of our construction was to incorporate the gradient flow intuition over $[0,T/2]$ of \eqref{eq:clerc-bdd} in a way that respected both the initial and terminal endpoint distributions. 

Lastly, we wish to emphasize the role of the time $T/2$ and its relation to the \textbf{turnpike property}, an important phenomenon in stochastic control \cite{trelat-turnpike-15,conforti-hjb-turnpike} that has been studied in the context of the $\Schro$ bridge problem as well as its kinetic and mean field variants in \cite{clerc23-longtime,entropic-turnpike-22,conforti21deriv,mean-field-sb-turnpike20}. The intuition provided by this phenomenon states that at large time, the dynamics of the $\Schro$ bridge away from its initial and terminal time endpoints is governed by the reference dynamics. To be more concrete, for some $\delta > 0$ and large $T >0$, the dynamics of $P^T$ restricted to $[T/2-\delta,T/2+\delta]$ should be very close to those of $R^T$ over the same interval. The presence of such behavior along the entropic interpolation is quantified in the results \cite[Sections 4.2.1, 4.3.3]{clerc23-longtime}. For our analysis, this phenomenon manifests in the fact that at time $T/2$, the time marginal of the $\Schro$ bridge becomes very close to the stationary measure $\diffref$ (see Proposition \ref{prop:chi-sqrd-bdd}).

\subsection{Outline of Paper.} In Section \ref{sec:preliminaries} we introduce the essential preliminary elements to prove our results. Section \ref{sec:ou-diff} proves the main diffusion approximation result when the ambient space is $\mathbb{R}^{d}$ and the reference process is \eqref{eq:OU}. An analogous result for compact manifolds is then developed in Section \ref{sec:manifold}. Most proofs are relegated to the Appendix. 

\subsection{AI Disclosure.} The author acknowledges the use of ChatGPT 5.6 Sol in the preparation and editing of the manuscript. ChatGPT identified and fixed an error in a first draft, specifically regarding the connection between the $\mathrm{CD}(\kappa,\infty)$ condition and the spectrum of the generator of the reference process. It identified literature establishing the Wasserstein decay in Proposition \ref{prop:bgl-9-7-2} in the spectral setting and performed the Gaussian computation after Theorem \ref{thm:euclidean-case} (which the author then verified). Additionally, it identified and corrected minor errors, gaps, and typos in the proofs. The construction of the diffusion approximation and all theorem statements were original ideas of the author, who wrote the manuscript and retains all responsibility for correctness of the results.

\section{Preliminaries}\label{sec:preliminaries}
We now introduce the necessary preliminary information and notation to properly state and prove the results of this paper. 

\subsection{Diffusions on Manifolds}
Let $(M,g)$ denote a smooth ($C^{\infty}$), connected, complete Riemannian manifold without boundary, equipped with the Levi-Civita connection denoted $\nabla$. Let $d: M \times M \to [0,+\infty)$ denote the distance on $M$ induced by the metric. In Section \ref{sec:ou-diff}, $M = \mathbb{R}^{d}$ with its standard Euclidean metric, and in Section \ref{sec:manifold} we take $M$ to be compact. Lastly, let $C([0,T],M)$ denote the set of continuous functions $[0,T] \to M$, and denote the coordinate process by $(x_t, t \in [0,T])$.

To define the $\Schro$ bridge on $(M,g)$ compact, we must specify a reference process. For a proper introduction to diffusion processes on manifolds, we refer readers to \cite{hsu-stoch-analysis-manifold,bgl-markov}. The reference processes we consider have the following form. Let $(B_t^{M}, t \geq 0)$ denote Brownian motion on $M$, and let $U \in C^2(M)$ be a potential. The following manifold-valued SDE,
\begin{equation}\label{eq:reference-process-intro}
    dX_t = -\frac{1}{2}\nabla U(X_t) dt + dB_t^{M},
\end{equation}
is reversible with stationary measure $e^{-U(x)}\vol(dx)$, where $\vol(dx)$ is the volume measure on $M$. The reference measure is defined as 
\begin{align}\label{eq:ref-measure}
    \diffref(dx) := \exp(-U(x))\vol(dx),
\end{align}
where we insist that $U$ in \eqref{eq:reference-process-intro} is chosen such that $\diffref$ is a probability measure. 
We denote the semigroup corresponding to \eqref{eq:reference-process-intro} by $(R_t, t \geq 0)$, defined for $f \in L^2(\diffref)$ by
\begin{align}
    R_t f(x) &= \Exp{}\left[f(X_{t})|X_0 = x\right].
\end{align}
For each $t > 0$ and $x \in M$, the transition kernel $r_{t}(x,\cdot)$ is the Radon-Nikodym derivative of the law of $X_t|X_0 = x$ with respect to $\diffref$, so that
\begin{align}
    R_t f(x) &= \int_{M} f(y) r_{t}(x,y) \diffref(dy). 
\end{align}
By reversibility, each $r_t$ is symmetric in its arguments. Note that the adjoint of the semigroup acts on measures. In particular, if $d\mu/d\diffref = \rho$ then $R_{t}^*\mu(dx) = (R_{t}\rho)\diffref(dx)$.

\subsection{$\Schro$ Bridge.} We define the $\Schro$ bridge on manifolds. Analogous to \eqref{def:sb-dyn-ou}, fix $T > 0$ and set $R^{T} = \mathrm{Law}(X_t, t \in [0,T])$ where $(X_t, t \geq 0)$ satisfies \eqref{eq:reference-process-intro} with $X_0 \sim \diffref$. The working assumption for the paper is the following.
\begin{assumption}\label{assumption}
    Let $(X_t, t \geq 0)$ be a reference process that is either (1) \eqref{eq:OU} when $M = \mathbb{R}^{d}$ or (2) as defined in \eqref{eq:reference-process-intro} with $U \in C^2(M)$ when $M$ is compact.
    Let $\mu, \nu \in \cP_2(M)$ with $H(\mu|\diffref), H(\nu|\diffref)$ finite. Let $\rho = d\mu/d\diffref$ and $\sigma = d\nu/d\diffref$ be such that $\rho, \sigma \in L^{\infty}(\diffref)$ with compact support. 
\end{assumption}
The (dynamic) T-$\Schro$ bridge from $\mu$ to $\nu$ is defined as
\begin{align}\label{eq:dynam-schro-bridge-defn}
     P^T := \argmin \limits_{P \in \Pi(\mu,\nu)} H(P|R^{T}).
\end{align}
The quantity $H(P^{T}|R^{T})$ is called the \textbf{entropic cost}. By \cite[Proposition 2.5]{schroLeonard13}, under Assumption \ref{assumption} the $\Schro$ bridge defined in \eqref{eq:dynam-schro-bridge-defn} exists, is unique, and possesses an $(f,g)$-decomposition. That is, there exists $a^{T},b^{T}: M \to [0,+\infty)$ measurable such that 
\begin{align}\label{eq:product-decomp}
    \frac{dP^{T}}{dR^{T}}(x) = a^{T}(x_0)b^{T}(x_{T}). 
\end{align}
Following the convention of \cite{conforti21deriv}, we insist that $\norm{b^{T}}_{L^1(\diffref)} = 1$ so that the pair $(a^{T},b^{T})$ is unique. We recall the following results, noting that \cite{conforti21deriv} proves them in a setting that includes, for instance, when the reference process satisfies $\mathrm{CD}(\kappa,\infty)$ for some $\kappa > 0$. This is satisfied when $\mathbb{R}^{d}$ is equipped with \eqref{eq:OU} reference. However, we wish to avoid curvature dimension assumptions in the compact manifold case, so we present an elementary proof of these results in the Appendix in Section \ref{subsec:compact-manifold-proofs}.
\begin{proposition}[Lemmas 2.1 and 3.6 in \cite{conforti21deriv}]\label{prop:schro-limit-facts}
    Under Assumption \ref{assumption}, $a^{T} \to \rho$ and $b^{T} \to \sigma$ in $L^2(\diffref)$ as $T \to +\infty$. Moreover, for all $T_0 > 0$ there exists $c_{T_0} > 0$ such that for all $T \geq T_0$
    \begin{align*}
        \norm{a^T}_{L^{\infty}(\diffref)} \leq \frac{1}{c_{T_0}}\norm{\rho}_{L^{\infty}(\diffref)}, \quad \norm{b^T}_{L^{\infty}(\diffref)}\norm{a^T}_{L^{1}(\diffref)} \leq \frac{1}{c_{T_0}}\norm{\sigma}_{L^{\infty}(\diffref)}.
    \end{align*}
\end{proposition}

Let $\mu_{t}^{T} = (x_t)_{\#}P^{T}$ denote the time $t \in [0,T]$ marginal distribution of the $\Schro$ bridge. Recall that the collection $(\mu_{t}^{T}, t \in [0,T])$ is called the $T$-\textbf{entropic interpolation}. Note that $d\mu_{t}^{T}/d\diffref = R_{t}a^{T}R_{T-t}b^{T}$. Importantly, as a consequence of the $L^2(\diffref)$ convergence of $a^{T}$ to $\rho$, it holds that $\int_{M} a^{T} \diffref \to 1$ as $T \to +\infty$. Hence, given $T_0 > 0$ large enough,
\begin{align}\label{eq:uniform-bdd-potents}
    \sup\limits_{T \geq T_0} \left(\norm{a^{T}}_{L^{\infty}(\diffref)}+\norm{b^{T}}_{L^{\infty}(\diffref)}\right) < +\infty \text{ and thus } \sup\limits_{T \geq T_0} \left(\sup\limits_{t \in [0,T]}\norm{\frac{d\mu_{t}^{T}}{d\diffref}}_{L^{\infty}(\diffref)} \right)<+\infty. 
\end{align}

\subsection{Optimal Transport.}\label{subsec:ot-prelim}
As outlined in the introduction, the construction of the diffusion approximation $Q^T$ in \eqref{defn:diffusion-approx} uses the quadratic cost optimal transport map at time $T/2$. This requires a brief introduction to the important notions of optimal transport. For a general treatise on optimal transport, we suggest the texts \cite{villani2021topics,santam2015ot,villani2009optimal}.

On a smooth manifold $(M,g)$, the squared 2-\textbf{Wasserstein} distance between $\mu, \nu \in \cP_2(M)$ is defined by
\begin{align}\label{eq:wass-dist}
    \Was{2}^2(\mu,\nu) &:= \inf \limits_{\gamma \in \Pi(\mu,\nu)} \int_{M \times M} d^2(x,y)\gamma(dxdy), 
\end{align}
where $\Pi(\mu,\nu)$ is the set of probability measures on $M\times M$ with first marginal equal to $\mu$ and the second equal to $\nu$. When $\mu$ is absolutely continuous with respect to $\vol$, the optimal choice of $\gamma$ in \eqref{eq:wass-dist} is unique; denote it $\gamma^*$. Moreover, it is induced by a transport map \cite[Theorem 10.41]{villani2009optimal}. That is, there exists a measurable function $T: M \to M$ such that $\gamma^* = (\mathrm{Id},T)_{\#}\mu$. When $M = \mathbb{R}^{d}$, \cite{THEbrenier} proved that there is a convex function $\phi$, unique up to a constant, such that $T = \nabla \varphi$. A similar characterization on manifolds was first shown by \cite{mccann-ot-maifold01}. We will call $T$ the \textbf{quadratic cost optimal transport map}.

\subsection{Spectral Analysis.}\label{subsec:spectral-analysis}
The reference diffusion in \eqref{eq:reference-process-intro} has an associated second order differential operator called the (infinitesimal) generator, defined for $f \in C_{c}^{\infty}(M)$ by
\begin{align}\label{defn:generator-of-diff}
    Lf := -\frac{1}{2}\langle \nabla U, \nabla f \rangle + \frac{1}{2}\Delta_g f,
\end{align}
where $\Delta_{g}$ denotes the Laplace-Beltrami operator on $(M,g)$. In the case of $\mathbb{R}^{d}$, $\Delta_g$ is the standard Laplacian $\Delta = \sum_{i=1}^{d} \partial^2_{i}$. 

On a compact manifold, it is a standard fact \cite{milman18-spectral} that the operator $L$ admits a spectral decomposition. That is, there is a discrete sequence of eigenvalues $(\lambda_{j}, j \geq 0) \subset \mathbb{R}$ with $0 = \lambda_0 < \lambda_1$ and $\lambda_{j} \leq \lambda_{j+1}$ for all $j \geq 0$ and $\lambda_j \to +\infty$ as $j \to+\infty$. We denote the corresponding sequence of eigenvectors as $(\psi_{j}, j \geq 0) \subset L^2(\diffref)$ with $\psi_{0} = 1$. The collection is chosen to form an orthonormal basis of $L^2(\diffref)$ such that $L\psi_j = -\lambda_{j}\psi_{j}$. Similarly, on $\mathbb{R}^{d}$ with the \eqref{eq:OU} generator, it is a standard fact that a spectral decomposition also exists. In this case, $\lambda_1 = 1/2$ and the eigenfunctions are formed from the Hermite polynomials. For more explicit examples, consult \cite[Section 2.7]{bgl-markov}.

Additionally, although $\lambda_0$ is also a simple eigenvalue (multiplicity one), this may not be the case for $\lambda_1$. We define
\begin{align}
    \mathcal{I}_{1} := \{i \geq 0: \lambda_i = \lambda_1\}.
\end{align}
Note that $\mathcal{I}_{1}$ is always a finite set. 

The presence of a discrete spectrum in our settings allows us to deduce more refined information in the large temperature limit. For instance, it is known from \cite[Theorem 1.4]{conforti21deriv} that $H(P^{T}|R^{T}) \to H(\mu|\diffref)+H(\nu|\diffref)$ as $T \to +\infty$ at a rate of $\exp(-\kappa T/2)$ that is shown to be sharp. In our current setting, we are able to exactly compute the next order term in terms of spectral quantities.  
\begin{theorem}\label{thm:next-term-entropic-cost}
    Let $(M,g)$ be either (1) $\mathbb{R}^{d}$ with \eqref{eq:OU} reference process, or (2) a compact manifold with reference process \eqref{eq:reference-process-intro} satisfying $U \in C^2(M)$. Let $\mu,\nu \in \cP_2(M)$ satisfy Assumption \ref{assumption}. The following limit holds
    \begin{align*}
        \lim\limits_{T \to +\infty} \exp(\lambda_1 T)\left(H(P^{T}|R^{T}) - \left(H(\mu|\diffref)+H(\nu|\diffref)\right)\right) &= -\sum_{i \in \mathcal{I}_{1}} p_iq_i,
    \end{align*}
    where $p_i = \langle \psi_i, \rho\rangle_{L^2(\diffref)}$ and $q_i = \langle \psi_i,\sigma \rangle_{L^2(\diffref)}$.
\end{theorem}

\section{Diffusion Approximation: $\mathbb{R}^{d}$ with \eqref{eq:OU} Reference}\label{sec:ou-diff}
We now properly state Main Result \ref{mainresult} in the Euclidean setting. 
\begin{theorem}\label{thm:euclidean-case}
    Equip $\mathbb{R}^{d}$ with reference process \eqref{eq:OU}, so that $\diffref$ is $N(0,\Id)$. Let $\mu,\nu \in \cP_2(\mathbb{R}^{d})$ satisfy Assumption \ref{assumption}. Let $P^{T}$ be the $\Schro$ bridge from $\mu$ to $\nu$ with \eqref{eq:OU} reference as defined in \eqref{eq:dynam-schro-bridge-defn}, and let $Q^{T}$ denote the noising-denoising approximation in \eqref{defn:diffusion-approx}. Then
    \begin{align}
        \limsup\limits_{T \to +\infty} \exp(T/2)\left(H(Q^T|P^T)+H(P^T|Q^T)\right) < +\infty. 
    \end{align}
\end{theorem}

To assess the tightness of this asymptotic we compute explicitly $P^{T}$ and $Q^{T}$ in the case of univariate Gaussian marginals. We note that this setting is not included in the theorem as $\mu$ and $\nu$ must be compactly supported. 

\textbf{Gaussian computations.} Because explicit formulas for the relevant quantities exist in the Gaussian setting, see \cite{zhang-sb-ouformula-2026,bunne-sbformula-26,janati2020} and Proposition \ref{prop:euc-gaussian-comp}, we let $\mu = N(0,1)$ and $\nu = N(0,\eta^2)$. To keep calculations simple, we only compute the endpoint distributions, i.e.\ we set $\pi^{T} := (x_0,x_{T})_{\#}P^{T}$ and $q^{T} := (x_0,x_{T})_{\#}Q^{T}$. Note that by the information processing inequality $H(q^{T}|\pi^{T}) \leq H(Q^{T}|P^{T})$ and $H(\pi^{T}|q^{T}) \leq H(P^{T}|Q^{T})$, so the above asymptotic still holds. In this case, $\pi^{T} \sim N(0,\Sigma_{1})$ and $q^{T} \sim N(0,\Sigma_2)$ with
    \begin{align}
        \Sigma_{1} = \begin{pmatrix} 1 & c_{T} \\ c_{T} & \eta^2 \end{pmatrix}, \quad \Sigma_2 = \begin{pmatrix} 1 & \frac{e^{-T/2}\eta^2}{\sqrt{1+e^{-T/2}(\eta^2-1)}} \\ \frac{e^{-T/2}\eta^2}{\sqrt{1+e^{-T/2}(\eta^2-1)}} & \eta^2 \end{pmatrix},
    \end{align}
    where $c_{T} = \frac{1}{2e^{-T/2}}\left(\sqrt{(1-e^{-T})^2+4e^{-T}\eta^2}-(1-e^{-T})\right)$. On the basis of these formulas, it holds as $T \to +\infty$ that
    \begin{align*}
         \lim\limits_{T \to +\infty }\exp(2T)\left(H(\pi^{T}|q^{T})+H(q^{T}|\pi^{T})\right) &< +\infty.
    \end{align*}
    Hence, the implied asymptotic rate for the endpoint distributions that Theorem \ref{thm:euclidean-case} suggests is not tight in the case of Gaussian marginals. Lastly, note that when $\eta^2 = 1$, i.e.\ when $\mu = \nu = \diffref$, all the relative entropy quantities vanish for all $T > 0$.
    We give the details of this Gaussian computation in Proposition \ref{prop:euc-gaussian-comp} in the Appendix. 

The proof of Theorem \ref{thm:euclidean-case} is relegated to the Appendix. The key advantage of performing computations using the symmetric relative entropy is that the potentials $a^{T}, b^{T}$ cancel, resulting in an expression that only involves quantities relating to the diffusion. A similar maneuver is used in the small-temperature analysis of \cite{AHMP25,MP25} and the estimates of \cite[Theorem 1.2]{chiarini2022gradient}.

The key nontrivial step in the Euclidean setting considered in Theorem \ref{thm:euclidean-case} concerns the regularity of the quadratic cost optimal transport map between $R_{T/2}^*\mu$ and $R_{T/2}^*\nu$, denoted $\nabla \varphi_{T/2}$. This relies on quantitative results on the propagation of log-concavity along \eqref{eq:OU}. We summarize this in the following proposition, and delay the proof to the Appendix. 
\begin{proposition}\label{prop:prop-of-log-conc}
    Let $\mu,\nu \in \cP_2(\mathbb{R}^{d})$ satisfy Assumption \ref{assumption}. Then for some $T_0 > 0$ large enough, there is $C >0$ such that $\norm{\nabla^2 \varphi_{T/2}} \leq C$ for all $T \geq T_0$. In particular, this implies the existence of a different constant $C > 0$ such that
    \begin{align*}
        \norm{\nabla \varphi_{T/2}(x)} \leq C(1 + \norm{x}), \quad \text{for all $T \geq T_0$ and $x \in \mathbb{R}^{d}$.}
    \end{align*}
\end{proposition}

We note that the propagation of various notions related to log-concavity along \eqref{eq:OU} is a rich area of study, with recent work including \cite{chaintron2025propagation,brigati-heatflow25,silveri-logconcav-25}. Indeed, the above result can be extended to $\mu, \nu$ such that their negative-log Lebesgue density is a Lipschitz perturbation of a strictly convex function, see \cite[Proposition 2.10]{chaintron2025propagation} and \cite[Theorem 1.3]{brigati-heatflow25}. However, we only consider the compact support setting as we will require the results of \cite[Lemma 2.1, Lemma 3.6]{conforti21deriv} summarized in Proposition \ref{prop:schro-limit-facts}, which require compact support. 

It is important to note that the creation of log-concavity (in the strong sense of uniform Hessian bounds) is unique to \eqref{eq:OU}-like processes \cite{kol-ou-logcov-01}. The proof of Proposition \ref{prop:ou-density-ratio} in the Appendix strongly uses the fact that the construction of $Q^T$ employs the quadratic cost optimal transport map. In general, the Lipschitz regularity of optimal transport beyond log-concavity is very difficult to establish. This suggests that a different approach may have to be employed to extend Theorem \ref{thm:euclidean-case} to non-\eqref{eq:OU} reference processes, perhaps using the non-optimal but Lipschitz Kim-Milman transport map as developed in \cite{kim-millman-map12} and whose Lipschitz properties were studied in \cite{msLipshiz23,fms-lip-24,ce2025-lip}.

\section{Diffusion Approximation: Compact Manifold}\label{sec:manifold}
Throughout this section, let $(M,g)$ be a compact manifold with the reference process \eqref{eq:reference-process-intro} and reference measure $\diffref \in \cP(M)$ defined in \eqref{eq:ref-measure} and related spectral quantities as defined in Section \ref{sec:preliminaries}.

First, we extend the diffusion approximation \eqref{defn:diffusion-approx} to the manifold setting. Indeed, the definition is the same formula symbol-for-symbol, with each quantity being replaced by its manifold quantity as defined in Section \ref{sec:preliminaries}. The only modification is with regard to the optimal transport map. Here, we let $\Phi_{T/2}: M \to M$ denote the optimal transport map from $R_{T/2}^*\mu$ to $R_{T/2}^*\nu$ with respect to the cost $d^2(\cdot,\cdot)$, which is guaranteed to exist under our hypotheses as explained in Section \ref{subsec:ot-prelim}. Altogether then,
\begin{align}\label{defn:diffusion-approx-manifold}
    \frac{dQ^{T}}{dR^{T}}(x) &:= \frac{\rho(x_0)\sigma(x_T)}{R_{T/2}\sigma(\Phi_{T/2}(x_{T/2}))} \cdot \frac{r_{T/2}(\Phi_{T/2}(x_{T/2}),x_{T})}{r_{T/2}(x_{T/2},x_{T})}
\end{align} 
With this definition, Main Result \ref{mainresult} in this setting becomes
\begin{theorem}\label{thm:compact-manifold}
    Let $(M,g)$ be a compact manifold with reference process having a spectral gap $\lambda_1> 0$. For $\mu,\nu \in \cP_2(M)$ satisfying Assumption \ref{assumption},   
    \begin{align}
        \limsup\limits_{T \to +\infty} \exp(\lambda_{1} T)(H(Q^{T}|P^{T})+H(P^{T}|Q^{T})) < +\infty. 
    \end{align}
\end{theorem}
Because $M$ is compact, there is no need to control the regularity of the optimal transport maps $(\Phi_{T/2}, T > 0)$ as in Proposition \ref{prop:prop-of-log-conc}. This gives a simplified argument.

Lastly, we show that this rate of approximation cannot in general be improved. To do this, we specialize to the same-marginal case (i.e.\ $\mu = \nu$) to simplify the argumentation. Indeed, in this case we can precisely calculate the rescaled limit in terms of the spectral decomposition of the density $\rho$. 
\begin{theorem}\label{thm:compact-manifold-same}
    Retain the setting of Theorem \ref{thm:compact-manifold}. When $\mu = \nu$, it holds that
    \begin{align}
        \lim\limits_{T \to +\infty} \exp(\lambda_1 T)(H(Q^{T}|P^{T})+H(P^{T}|Q^{T})) &= \sum_{i \in \mathcal{I}_{1}}p_i^2,
    \end{align}
    where $p_i = \langle \rho, \psi_i \rangle_{L^2(\diffref)}$. 
\end{theorem}
A benefit of Theorem \ref{thm:compact-manifold-same} is that it highlights that the proof approach for Theorems \ref{thm:euclidean-case} and \ref{thm:compact-manifold} cannot be meaningfully improved. Indeed, the sum of the two relative entropies in the two theorems results in two terms that we analyze separately in the Appendix: a diffusion term in Proposition \ref{prop:first-half-symm-rel-ent}, and a heat kernel comparison term in Propositions \ref{prop:heat-kernel-log-bdd} and \ref{prop:ou-density-ratio}. The heat kernel comparison arguments are rather loose, but Theorem \ref{thm:compact-manifold-same} demonstrates that any sufficient tightening of these argument will not in general improve the quantitative exponential decay of the relative entropy approximation $Q^{T}$ gives to $P^{T}$. 

\section{A Large Time Decomposition Property}\label{sec:large-time-semgroup}
We now state and prove the structure result for $\Schro$ bridges discussed in the paragraph after the statement of Main Result \ref{mainresult}.

First, we fix notation and explain the construction. Let $(M,g)$ be a compact smooth Riemannian manifold. Let $\mu,\nu \in \cP_{2}(M)$ satisfy Assumption \ref{assumption}, and fix $T > 0$. Let $P_{\mu}^{T/2}, P_{\nu}^{T/2}$ denote the $\Schro$ bridge with respect to $R^{T/2}$ from $\mu$ to $\diffref$ and from $\diffref$ to $\nu$, respectively. We write the following product decompositions, following the convention that $\norm{b^{T/2}}_{L^{1}(\diffref)} = \norm{\beta^{T/2}}_{L^{1}(\diffref)} = 1$:
\begin{align}
    \frac{dP_{\mu}^{T/2}}{dR^{T/2}} = a^{T/2}(x_0)b^{T/2}(x_{T/2}), \quad \frac{dP_{\nu}^{T/2}}{dR^{T/2}} &= \alpha^{T/2}(x_0)\beta^{T/2}(x_{T/2}). 
\end{align}
Let $\gamma_{T/2}(\cdot|x_{0})$ denote the conditional distribution of $x_{T/2}|x_0$ under $P_{\nu}^{T/2}$. Then it has density with respect to $\diffref$ equal to
\begin{align}
    \frac{\gamma_{T/2}(\cdot|x_{0})}{d\diffref}(x_{T/2}) &= \frac{\beta^{T/2}(x_{T/2})r_{T/2}(x_{0},x_{T/2})}{R_{T/2}\beta^{T/2}(x_{0})}.
\end{align}
Now we construct a path measure $Q^{T}$ by gluing together these two $\Schro$ bridges along their common marginal $\diffref$. First, take $P_{\mu}^{T/2}$ as is, then define $x_{T}|x_{T/2}$ to have the distribution $\gamma_{T/2}$. This defines a triplet of random variables $(X_0,X_{T/2},X_{T})$ with marginals $\mu,\diffref,\nu$, respectively, and joint density with respect to $\diffref^{\otimes 3}$ equal to
\begin{align}
    a^{T/2}(x_0)b^{T/2}(x_{T/2}) \cdot \frac{\beta^{T/2}(x_{T})}{R_{T/2}\beta^{T/2}(x_{T/2})} \cdot r_{T/2}(x_0,x_{T/2})r_{T/2}(x_{T/2},x_{T}).
\end{align}
We define $Q^{T}$ as
\begin{align}\label{eq:semigrp-sb-approx}
    \frac{dQ^{T}}{dR^{T}} &:= a^{T/2}(x_0) b^{T/2}(x_{T/2})  \cdot \frac{\beta^{T/2}(x_{T})}{R_{T/2}\beta^{T/2}(x_{T/2})}.
\end{align}
The following result quantifies the large time approximation $Q^{T}$ provides to the $\Schro$ bridge at temperature $T$ from $\mu$ to $\nu$. 
\begin{theorem}\label{thm:semigrp-large-temp}
    Retain the setting of Theorem \ref{thm:compact-manifold}. Fix $\mu,\nu \in \cP_2(M)$ satisfying Assumption \ref{assumption}. Let $P^{T}$ denote the $\Schro$ bridge from $\mu$ to $\nu$ computed with reference $R^{T}$, and let $Q^{T}$ denote the path measure defined in \eqref{eq:semigrp-sb-approx}. It holds that
    \begin{align}
        \limsup\limits_{T \to +\infty} \exp(\lambda_1 T)\left(H(Q^{T}|P^{T})+H(P^{T}|Q^{T})\right) < +\infty.
    \end{align}
\end{theorem}

\begin{proof}
    Recall the product structure that $P^{T}$ has with respect to $R^{T}$. Since $x_0$ and $x_{T}$ have the same marginal distributions under $P^{T}$ and $Q^{T}$,
\begin{align}
    H(Q^{T}|P^{T})+H(P^{T}|Q^{T}) &= \left(\Exp{Q^{T}}-\Exp{P^{T}}\right)\left[\log\left(\frac{b^{T/2}(x_{T/2})}{R_{T/2}\beta^{T/2}(x_{T/2})}\right)\right].
\end{align}
Moreover, $x_{T/2}$ is distributed as $\diffref$ under $Q^{T}$ and $\mu_{T/2}^{T}$ under $P^{T}$, giving that
\begin{align}
     H(Q^{T}|P^{T})+H(P^{T}|Q^{T}) &= \Exp{\diffref}\left[\left(1-\frac{d\mu_{T/2}^{T}}{d\diffref}\right)\log\left(\frac{b^{T/2}(x_{T/2})}{R_{T/2}\beta^{T/2}(x_{T/2})}\right)\right].
\end{align}
The $\Schro$ system gives that $1 = R_{T/2}a^{T/2}b^{T/2}$, so applying Cauchy-Schwarz gives that
\begin{align}
    H(Q^{T}|P^{T})+H(P^{T}|Q^{T}) \leq \sqrt{\chi^2(\mu_{T/2}^T||\diffref)} \cdot \norm{\log(R_{T/2}\beta^{T/2})+\log(R_{T/2}a^{T/2})}_{L^2(\diffref)}
\end{align}
By Proposition \ref{prop:chi-sqrd-bdd}, $\sqrt{\chi^2(\mu_{T/2}^T||\diffref)}$ decays with $\exp(-\lambda_1 T/2)$. We will further argue that the $L^2(\diffref)$-norm terms are both $\exp(-\lambda_1 T/2)$. To see this, recall the numerical inequality $\frac{x}{x+1} \leq \log(1+x) \leq x$ for all $x > -1$. Moreover, as $1 = R_{T/2}a^{T/2}\beta^{T/2}$ and $1 = \alpha^{T/2} R_{T/2}\beta^{T/2}$, it holds that
\begin{align*}
    \abs{\log(R_{T/2}\beta^{T/2})} \leq (\alpha^{T/2}+1)\abs{R_{T/2}\beta^{T/2}-1}, \quad \abs{\log(R_{T/2}a^{T/2})} \leq (b^{T/2}+1)\abs{R_{T/2}a^{T/2}-1}.
\end{align*}
By Proposition \ref{prop:schro-limit-facts}, $\norm{\alpha^{T/2}}_{L^{\infty}(\diffref)}, \norm{b^{T/2}}_{L^{\infty}(\diffref)}$ stay bounded as $T \to +\infty$. Repeating the arguments from Theorem \ref{thm:next-term-entropic-cost} establish that $\norm{R_{T/2}\beta^{T/2}-1}_{L^2(\diffref)}, \norm{R_{T/2}a^{T/2}-1}_{L^2(\diffref)}$ both decay at rate $\exp(-\lambda_1 T/2)$, completing the argument. 
\end{proof}

\section{Appendix}
Before presenting the proofs, we fix notation and recall some important results. Throughout all proofs, Assumption \ref{assumption} is in effect as well as all the notation developed in Section \ref{sec:preliminaries}. 

\subsection{Notation and Preliminary Results}
Recall $a^T, b^T: M \to [0,+\infty)$ as defined in \eqref{eq:product-decomp} with the convention that $\int_{M} b^{T} d\diffref = 1$. By Proposition \ref{prop:schro-limit-facts}, $a^{T}, b^T \in L^2(\diffref)$ and thus there exist sequences $(\alpha_{i}(T),i \geq 1)$ and $(\beta_{i}(T), i \geq 1)$ such that in $L^2(\diffref)$
\begin{align}\label{eq:spectral-expn-potents}
    a^{T} = \left(\int_{M} a^{T}d\diffref\right)+\sum_{i=1}^{\infty} \alpha_i(T) \psi_i, \quad b^{T} = 1+\sum_{i=1}^{\infty} \beta_i(T) \psi_i
\end{align}
As $\rho,\sigma \in L^2(\diffref)$ as well, there exist sequences $(p_i, i \geq 1)$ and $(q_i, i \geq 1)$ such that in $L^2(\diffref)$
\begin{align}
    \rho = 1 + \sum_{i=1}^{\infty} p_{i}\psi_i, \quad \sigma = 1 + \sum_{i=1}^{\infty} q_i \psi_{i}.
\end{align}

Moreover, by the $L^2(\diffref)$ convergence in Proposition \ref{prop:schro-limit-facts}, it holds that
\begin{align}
    \lim\limits_{T \to +\infty} \left(\int_{M} a^{T} d\diffref-1\right)^2 + \sum_{i=1}^{\infty} (\alpha_{i}(T)-p_{i})^2= \lim\limits_{T \to +\infty} \sum_{i=1}^{\infty} (\beta_{i}(T)-q_{i})^2 = 0. 
\end{align}

Before beginning the relevant proofs, we recall two basic facts that will be applied several times throughout the proofs. Both of these facts are consequences of $\mathrm{CD}(\kappa,\infty)$, although to avoid introducing this assumption in the compact manifold case we present proofs using spectral analysis. This approach results in weaker bounds that retain the desired asymptotic decay. 
\begin{proposition}[Heat Flow Decay]\label{prop:bgl-9-7-2}
    Let $(R_{t}, t \geq 0)$ be a diffusion semigroup, and let $\mu$ and $\nu$ satisfy Assumption \ref{assumption}. Then there exist $C > 0$ and $T_0 > 0$ (both depending on $\mu,\nu$) such that for all $t \geq T_0$,
    \begin{align*}
        \Was{2}(R_{t}^*\mu,R_{t}^*\nu) \leq C\exp(-\lambda_1 t). 
    \end{align*}
\end{proposition}
\begin{proof}
    In the case of \eqref{eq:OU}, this follows from \cite[Theorem 9.7.2]{bgl-markov} with constant $C = \Was{2}(\mu,\nu)$. In the case of $(M,g)$ compact, we argue directly from the spectral decomposition. By \cite[Theorem 1]{peyre2018comparison}, it holds for all $\mu_1,\mu_2 \in \cP(M)$ that
    \begin{align}\label{eq:w2-neg-sob-comp}
        \Was{2}(\mu_1,\mu_2) \leq 2 \norm{\mu_1-\mu_2}_{\dot{H}^{-1}(\mu_1)},
    \end{align}
    where for any signed measure $\rho$ on $M$ such that $\int_{M} d\rho = 0$,
    \begin{align}\label{defn:neg-sob-norm}
        \norm{\rho}_{\dot{H}^{-1}(\mu_1)} &:= \sup\left\{ \int_M f d\rho : f \in C^1(M), \int_{M} \norm{\nabla f}^2 d\mu_1 = 1\right\}.
    \end{align}
    Now, recall that $R_{t}^* \mu = R_{t}\rho \diffref$. Thus, there exists $T_0$ and $0 < c < C$ such that for all $t \geq T_0$, $c \diffref \leq R_{t}^*\mu \leq C\diffref$. Fix $t \geq T_0$ and let $f \in C^1(M)$ be such that $\int_{M} \norm{\nabla f}^2 dR_{t}^{*}\mu = 1$. Note that this implies that $c\norm{\nabla f}_{L^2(\diffref)}^2  \leq \int_{M} \norm{\nabla f}^2 dR_{t}^{*}\mu = 1\leq C\norm{\nabla f}_{L^2(\diffref)}^2 $. We now see that
    \begin{align*}
        \int_{M} f d(R_t^*\mu - R_t^* \nu) &= \int_{M} f R_{t}(\rho-\sigma) d\diffref = \int_{M} (R_t f) (\rho-\sigma)d\diffref,
    \end{align*}
    where the last equality follows from the symmetry of the semigroup on $L^2(\diffref)$. Since $\int_{M} (\rho -\sigma) d\diffref = 0$, we can without loss of generality here insist that $\int_{M} f d\diffref = 0$. By Cauchy-Schwarz then, it holds that
    \begin{align*}
        \abs{\int_{M} f d(R_t^*\mu - R_t^* \nu)} &\leq \norm{\rho-\sigma}_{L^2(\diffref)}\norm{R_t f}_{L^2(\diffref)} \\
        &\leq \norm{\rho-\sigma}_{L^2(\diffref)} \cdot (2\lambda_1)^{-1/2}\exp(-\lambda_1 t) \norm{\nabla f}_{L^2(\diffref)} \\
        &\leq \norm{\rho-\sigma}_{L^2(\diffref)} \cdot (2\lambda_1 c)^{-1/2}\exp(-\lambda_1 t),
    \end{align*}
    where the second line follows from the Poincaré inequality (see the paragraph after \cite[Proposition 3.1.6]{bgl-markov}). 
    As this bound holds for all $f$ in the definition of \eqref{defn:neg-sob-norm}, the proposition follows by \eqref{eq:w2-neg-sob-comp}.
\end{proof}

We will also make use of the following pointwise inequality. A similar inequality follows immediately under the $\mathrm{CD}(\kappa,\infty)$ condition and is presented in \cite[Proposition 8.6.1]{bgl-markov}. To avoid this assumption in the compact manifold case, we again present an argument from spectral analysis, at the cost of a less sharp leading constant. 
\begin{proposition}\label{prop:bgl-8-6-1}
    Let $f: M \to \mathbb{R}$ be a bounded and measurable function. Then there exists some $T_0 > 0$ and $C > 0$ such that for all $t \geq T_0$,
    \begin{align*}
        \sup\limits_{x \in M}\norm{\nabla R_{t}f(x)} \leq C\exp(-\lambda_1 t)\norm{f}_{L^\infty(\diffref)}
    \end{align*}
\end{proposition}
\begin{proof}
    Under OU, by the pointwise reverse Poincaré inequality developed in \cite[Equation (4.7.3)]{bgl-markov} for $f \in L^{\infty}(\diffref)$ and for all $t > 0$
\begin{align}
    \norm{\nabla R_{t}f}^2 \leq \frac{e^{-t}}{1-e^{-t}}(R_{t}f^2-(R_tf)^2).
\end{align}
This then implies the following inequality
\begin{align}
    \norm{\nabla R_{t}f}_{L^{\infty}(\diffref)} \leq \frac{\exp(-t/2)}{\sqrt{1-\exp(-t)}} \norm{f}_{L^{\infty}(\diffref)}. 
\end{align}
Next, consider the compact manifold case. Fix some $T_0 > 0$, then since $M$ is compact it holds that $C_{T_0} := \sup\limits_{x,y \in M} \norm{\nabla_{x} r_{T_0}(x,y)} < +\infty$. Set $f_0 := f - \int_{M} f d\diffref$, then for $t \geq T_0$
\begin{align*}
    \norm{\nabla R_{t}f_0(x)} &= \norm{\nabla R_{T_0}(R_{t-T_0}f_{0})(x)} \leq C_{T_0}\norm{R_{t-T_0}f_0}_{L^2(\diffref)} \leq C_{T_0} \exp(\lambda_1 (T_0-t))\norm{f_0}_{L^2(\diffref)}.
\end{align*}
Since $\nabla R_{t}f_{0} = \nabla R_{t}f$ and $\norm{f_0}_{L^{2}(\diffref)} \leq \norm{f}_{L^2(\diffref)}$, the proposition then follows from setting $C = C_{T_0} \exp(\lambda_{1} T_0)$ and the fact that $\norm{f}_{L^2(\diffref)} \leq \norm{f}_{L^{\infty}(\diffref)}$. 
\end{proof}

On the basis of the spectral expansion, we establish the following limit.
\begin{proposition}\label{prop:chi-sqrd-bdd}
    Under Assumption \ref{assumption}, it holds that 
    \begin{align}
        \limsup\limits_{T \to +\infty} \exp(\lambda_1 T)\chi^2(\mu_{T/2}^{T}||\diffref) &\leq 2(K^2+1)\sum_{i\in \mathcal{I}_1} (p_i^2 + q_i^2),
    \end{align}
    where $K = \sup\limits_{T \geq T^*} \norm{b^{T}}_{L^{\infty}(\diffref)}$ for some $T^* > 0$ chosen large enough, guaranteed to exist by \eqref{eq:uniform-bdd-potents}.
\end{proposition}

\begin{proof}
    Note that $\frac{d\mu_{T/2}^{T}}{d\diffref} = R_{T/2}a^{T} R_{T/2}b^{T} \in L^2(\diffref)$. As $\int_{M} R_{T/2}a^{T}R_{T/2}b^{T} d\diffref = 1$, from the spectral expansions of $a^{T}$ and $b^T$ we obtain that 
    \begin{align}
        1 &= \langle R_{T/2}a^{T} R_{T/2}b^{T}, 1 \rangle_{L^2(\diffref)} = \left(\int_{M} a^{T}d\diffref\right) + \sum_{i=1}^{\infty} \exp(-\lambda_{i} T)\alpha_{i}(T)\beta_{i}(T).  
    \end{align}
    It then follows that
    \begin{align}\label{eq:limit-l1-norm-aT}
        \lim\limits_{T \to +\infty} \exp(\lambda_1 T)\left(1-\int_{M} a^{T} d\diffref\right) &= \sum_{i \in \mathcal{I}_{1}}p_i q_{i}.
    \end{align}
    Hence, we can deduce
    \begin{align*}
        \lim\limits_{T \to +\infty}\exp(\lambda_{1}T)\norm{R_{T/2}a^{T}-1}_{L^2(\diffref)}^2 &= \lim\limits_{T \to +\infty}\exp(\lambda_{1}T)\left[\left(\int_{M} a^{T} d\diffref-1\right)^2 + \sum_{i=1}^{\infty} \exp(-\lambda_{i}T)\alpha_{i}^2(T)\right] \\
        &= \sum_{i \in \mathcal{I}_1} p_i^2. 
    \end{align*}
    More straightforwardly, $\lim\limits_{T \to +\infty}\exp(\lambda_1 T)\norm{R_{T/2}b^{T}-1}_{L^2(\diffref)}^2 = \sum_{i \in \mathcal{I}_1} q_i^2$. To combine these facts, we write
    \begin{align*}
        R_{T/2}a^{T}R_{T/2}b^{T}-1 &= \left((R_{T/2}a^{T}-1)R_{T/2}b^{T}\right) + \left(R_{T/2}b^{T}-1\right).
    \end{align*}
    Applying the triangle inequality establishes
    \begin{align*}
        \limsup\limits_{T \to +\infty}\exp(\lambda_1 T)\norm{R_{T/2}a^{T}R_{T/2}b^{T}-1}_{L^2(\diffref)}^2 \leq 2(K^2+1)\sum_{i\in \mathcal{I}_1} (p_i^2 + q_i^2). 
    \end{align*}
    Finally then, recall that $\chi^{2}(\mu_{T/2}^{T}||\diffref) = \norm{d\mu_{T/2}^{T}/d\diffref-1}^{2}_{L^2(\diffref)}$, establishing the claim.
\end{proof}

We now prove Theorem \ref{thm:next-term-entropic-cost}, computing the next order term in the entropic cost. 
\begin{proof}[Proof of Theorem \ref{thm:next-term-entropic-cost}]
    From the product structure of the $\Schro$ bridge, this requires us to compute the limit as $T \to +\infty$ of the exponential rescaling of 
    \begin{align*}
        \int_{M} \left(\log a^{T} - \log \rho\right)  d\mu + \int_{M} \left(\log b^{T}-\log \sigma \right) d\nu  =\left(-\int_{M}  \log R_{T}b^{T} d\mu - \int_{M} \log R_{T}a^{T} d\nu\right)
    \end{align*}
    where the second identity follows from the fact that $\rho = a^{T}R_{T}b^{T}$ and $\sigma = R_{T}a^{T} b^{T}$. 
    
    We now compute the two integrals separately. Recall the numerical inequality $\frac{x}{x+1} \leq \log (1+x) \leq x$ for all $x > -1$. It then holds that
    \begin{align}
        \int_{M} \frac{(R_{T}b^{T}-1)}{R_{T}b^{T}} d\mu \leq \int_{M} \log R_{T}b^{T} d\mu \leq \int_{M} \left(R_{T}b^{T}-1\right)d\mu.
    \end{align}
    First, work with the right hand side. Recall the spectral expansion of $b^{T}$ in $L^{2}(\diffref)$. With some rearrangement, the following representation holds in $L^2(\diffref)$
    \begin{align*}
        \exp(\lambda_1 T)(R_{T}b^{T}-1) - \left(\sum_{i \in \mathcal{I}_{1}} \beta_{i}(T)\psi_{i}\right) = \sum_{i \notin \left(\mathcal{I}_{1} \cup \{0\}\right)}^{+\infty} \exp(-(\lambda_i -\lambda_1)T)\beta_{i}(T)\psi_{i}. 
    \end{align*}
    The right hand side vanishes as $T \to +\infty$ in $L^2(\diffref)$. As $\rho \in L^{\infty}(\diffref)$, the right hand side also vanishes in $L^2(\rho)$. This establishes that
    \begin{align}
        \lim\limits_{T \to +\infty }\exp(\lambda_1 T)\int_{M}(R_{T}b^{T}-1)d\mu &=  \sum_{i \in \mathcal{I}_{1}} q_{i}\int_{M}\psi_{i}d\mu = \sum_{i \in \mathcal{I}_{1}} p_{i}q_{i}
    \end{align}
    For the left hand side, as $\rho = a^{T}R_{T}b^{T}$ we obtain
    \begin{align*}
        \int_{M} \frac{(R_{T}b^{T}-1)}{R_{T}b^{T}} d\mu &= \int_{M} (R_{T}b^{T}-1)a^{T} d\diffref = \int_{M} (R_{T}b^{T}-1)d\mu + \int_{M} (R_{T}b^{T}-1)(a^{T}-\rho)d\diffref. 
    \end{align*}
    We have already computed the rescaled limit of the first term on the rightmost expression. For the second term, as $\exp(\lambda_1 T)\norm{R_{T}b^{T}-1}_{L^{2}(\diffref)}$ stays bounded as $T \to +\infty$ and $\norm{a^{T}-\rho}_{L^2(\diffref)}$ vanishes as $T \to +\infty$, the term vanishes by Cauchy-Schwarz. Hence, it holds that
    \begin{align*}
        \lim\limits_{T \to +\infty}\exp(\lambda_{1}T)\int_{M} \log R_{T}b^{T} d\mu &=  \sum_{i \in \mathcal{I}_1} p_{i}q_{i}.
    \end{align*}
    The argument for the second integral is similar, with the only difference arising from the fact that $a^{T}$ has not been normalized to integrate to $1$. Regardless, the analogous chain of reasoning from the numerical inequality for $\log$ gives that
    \begin{align*}
        \int_{M} (R_{T}a^{T}-1)(b_{T}-\sigma) d\diffref \leq \int_{M} \log R_{T}a^{T}d\nu - \int_{M} \left(R_{T}a^{T}-1\right) d\nu \leq 0.
    \end{align*}
    With the spectral expansion of $a^{T}$ and the limit of $\int_{M} a^{T} d\diffref$ computed in \eqref{eq:limit-l1-norm-aT} (note the sign change), identical argumentation gives that
    \begin{align*}
        \lim\limits_{T \to +\infty}\exp(\lambda_{1}T)\int_{M} \log R_{T}a^{T} d\nu &= \lim\limits_{T \to +\infty}\exp(\lambda_{1}T) \left(\left(\int_{M}R_{T}a^{T}d\nu-\int_{M} a^{T}d\diffref\right) + \left(\int_{M}a^{T}d\diffref -1\right)\right) \\
        &= \sum_{i \in \mathcal{I}_{1}} \left(p_{i}q_{i}-p_{i}q_i\right) =0.
    \end{align*}
\end{proof}

The following proposition is key for the symmetric relative entropy calculation, and can be developed simultaneously in both the compact and non-compact cases we consider.
\begin{proposition}\label{prop:first-half-symm-rel-ent}
    In addition to Assumption \ref{assumption}, suppose that 
    \begin{align}\label{eq:distance-asmp}
        \limsup\limits_{T \to +\infty} \Exp{\diffref}\left[d^2(x,\Phi_{T/2}(x))\right] < +\infty,
    \end{align}
    where $\Phi_{T/2}: M \to M$ is the quadratic cost optimal transport map from $R_{T/2}^*\mu$ to $R_{T/2}^*\nu$. Then
    \begin{align*}
        \limsup\limits_{T \to +\infty}\exp(\lambda_1 T)\left(\left(\Exp{P_{T}}-\Exp{Q_{T}}\right)\left[\log\left(\frac{d(R_{T/2}^*\nu)}{d\diffref}(\Phi_{T/2}(x_{T/2}))\right)\right]\right) < +\infty.
    \end{align*}
    
\end{proposition}
\begin{remark}
    Note that \eqref{eq:distance-asmp} holds in the settings of Theorems \ref{thm:euclidean-case} and \ref{thm:compact-manifold}. When $(M,g)$ is a compact manifold, this claim is immediate as $M$ is bounded. In the Euclidean setting, recall that Proposition \ref{prop:prop-of-log-conc} gives the existence of some $C> 0$ such that for all $T$ large enough $\norm{\nabla \varphi_{T/2}(x)} \leq C(1+\norm{x})$ for all $x \in \mathbb{R}^d$. This completes the argument.
\end{remark}

\begin{proof}
    Under $Q_{T}$, the variable $\Phi_{T/2}(x_{T/2}) \sim R_{T/2}^*\nu$. As $R_{T/2}^*\nu, \diffref \in \cP(M)$, it holds that
\begin{align*}
    \Exp{Q_{T}}\left[\log\left(\frac{d(R_{T/2}^*\nu)}{d\diffref}(\Phi_{T/2}(x_{T/2}))\right)\right] = H(R_{T/2}^*\nu|\diffref) \geq 0. 
\end{align*}
Combining this with the trivial bound $\log(1+x) \leq x$ for all $x > -1$,
\begin{align*}
     &\left(\Exp{P_{T}}-\Exp{Q_{T}}\right)\left[\log\left(\frac{d(R_{T/2}^*\nu)}{d\diffref}(\Phi_{T/2}(x_{T/2}))\right)\right] \leq \Exp{\diffref}\left[\frac{d\mu_{T/2}^{T}}{d\diffref}(x)\left(R_{T/2}\sigma(\Phi_{T/2}(x))-1 \right)\right]\\
     &=\Exp{\diffref}\left[\left(\frac{d\mu_{T/2}^{T}}{d\diffref}(x)-1\right)\left(R_{T/2}\sigma(x)-1 \right)\right]+\Exp{\diffref}\left[\frac{d\mu_{T/2}^{T}}{d\diffref}(x)\left(R_{T/2}\sigma(\Phi_{T/2}(x))-R_{T/2}\sigma(x) \right)\right],
\end{align*}
where the equality follows from noting that $\Exp{\diffref}[R_{T/2}\sigma] = 1$. We now show each term is bounded in absolute value when we multiply by $\exp(\lambda_{1}T)$ and send $T \to +\infty$. First, by Cauchy-Schwarz
\begin{align*}
    &\exp(\lambda_{1}T)\abs{\Exp{\diffref}\left[\left(\frac{d\mu_{T/2}^{T}}{d\diffref}(x)-1\right)\left(R_{T/2}\sigma(x)-1 \right)\right]} \\
    &\leq \exp(\lambda_{1}T/2)\sqrt{\chi^2(\mu_{T/2}^{T}||\diffref)} \cdot \exp(\lambda_{1}T/2)\norm{R_{T/2}\sigma(x)-1}_{L^2(\diffref)},
\end{align*}
which is bounded as $T \to +\infty$ by Proposition \ref{prop:chi-sqrd-bdd} and the standard spectral expansion for $\sigma$. 

For the second term, recall from \eqref{eq:uniform-bdd-potents} that there exists $C > 0$ such that $d\mu_{T/2}^{T}/d\diffref \leq C$ for all $T$ large enough. Hence, for all $T$ large enough,
\begin{align*}
    \abs{\Exp{\diffref}\left[\frac{d\mu_{T/2}^{T}}{d\diffref}(x)\left(R_{T/2}\sigma(\nabla \varphi_{T/2}(x))-R_{T/2}\sigma(x) \right)\right]} \leq C\norm{\nabla R_{T/2}\sigma}_{L^{\infty}(\diffref)} \Exp{\diffref}\left[d(x,\nabla \varphi_{T/2}(x))\right].
\end{align*}
Since $\norm{\sigma}_{L^{\infty}(\diffref)} < +\infty$, by Proposition \ref{prop:bgl-8-6-1} it holds that $\norm{\nabla R_{T/2}\sigma}_{\infty}$ decays like $\exp(-\lambda_1 T/2)$ as $T \to +\infty$. For the other term, we write
\begin{align*}
    \Exp{\diffref}\left[d(x,\nabla \varphi_{T/2}(x))\right] &= \Exp{R_{T/2}^*\mu}\left[d(x,\nabla \varphi_{T/2}(x))\right]+\Exp{\diffref}\left[(1-R_{T/2}\rho)d(x,\nabla \varphi_{T/2}(x))\right]
\end{align*}
By Jensen's inequality, the $\Was{2}$ optimality of $\nabla \varphi_{T/2}$, and Proposition \ref{prop:bgl-9-7-2}, the first term of the RHS is bounded above by $C\exp(-\lambda_1 T/2)$.
On the other hand, by Cauchy-Schwarz 
\begin{align*}
    \abs{\Exp{\diffref}\left[(R_{T/2}\rho-1)d(x,\nabla \varphi_{T/2}(x))\right]} \leq \norm{R_{T/2}\rho-1}_{L^2(\diffref)}\left(\Exp{\diffref}[d^2(x,\nabla \varphi_{T/2}(x))]\right)^{1/2}.
\end{align*}
This term also decays like $\exp(-\lambda_1 T/2)$ as $T \to +\infty$ thanks to \eqref{eq:distance-asmp} and the spectral decay of $R_{T/2}\rho$. 
\end{proof}

\subsection{Compact Manifold Proofs}\label{subsec:compact-manifold-proofs}
First, as promised in Section \ref{sec:preliminaries} we present an elementary proof of Proposition \ref{prop:schro-limit-facts} in the compact manifold case that does not require a curvature dimension assumption.
\begin{proof}[Proof of Proposition \ref{prop:schro-limit-facts} in compact manifold setting]
As $M$ is compact, by \cite[Exercise (10.11)]{heat-kernel-manifold-grig09}, as $T \to +\infty$, $r_{t}(x,y) \to 1$ uniformly in $x,y \in M$. 
    That is, for each $T_0 >0$ there exist $0 < c_1 < 1 < C_1 < 2$ such that 
    \begin{align}\label{eq:uniform-bdd}
        c_1 \leq r_t(x,y) \leq C_1, \quad \text{for all $t \geq T_0$, $x,y \in M$.}
    \end{align}
    Since the $\Schro$ bridge is a coupling of $\mu$ and $\nu$, for all $T \geq T_0$ 
    \begin{align*}
        \rho(x) &= \int_{M} r_{T}(x,y)a^{T}(x)b^{T}(y) \diffref(dy) \geq c_1 a^{T}(x)\int_{M} b^{T}(y) \diffref(dy) \Rightarrow \norm{a^T}_{L^{\infty}(\diffref)} \leq \frac{1}{c_1} \norm{\rho}_{L^{\infty}(\diffref)}, \\
        \sigma(y) &= \int_{M} r_{T}(x,y)a^{T}(x)b^{T}(y) \diffref(dx) \geq c_1 b^{T}(y)\int_{M} a^{T}(x) \diffref(dx) \Rightarrow \norm{a^T}_{L^{1}(\diffref)}\norm{b^{T}}_{L^{\infty}(\diffref)} \leq \frac{1}{c_1} \norm{\sigma}_{L^{\infty}(\diffref)}.
    \end{align*}
    Next, observe for all $T \geq T_0$
    \begin{align*}
        \rho(x) \leq C_{1}a^{T}(x) \int_{M} b^{T}(y)\diffref(dy) \Rightarrow \int_{M} a^{T}(x) \diffref(dx) \geq C_{1}^{-1}. 
    \end{align*}
    Hence, we obtain that $\norm{a^{T}}_{L^{\infty}(\diffref)}, \norm{b^{T}}_{L^{\infty}(\diffref)}$ stay bounded as $T \to +\infty$.
    
    For the $L^2(\diffref)$ convergence, we note that
    \begin{align*}
        \norm{a^{T}-\rho}_{L^2(\diffref)} = \norm{a^{T}-a^{T}R_{T}b^{T}}_{L^2(\diffref)} \leq \norm{a^{T}}_{L^{\infty}(\diffref) }\norm{R_{T}b^{T}-1}_{L^{2}(\diffref)} \leq \norm{a^{T}}_{L^{\infty}(\diffref)} \cdot \exp(-\lambda_1 T)\norm{b^{T}}_{L^{2}(\diffref)},
    \end{align*}
    which vanishes as $T \to +\infty$ as we have now established that $\norm{a^{T}}_{L^{\infty}(\diffref)}, \norm{b^{T}}_{L^{\infty}(\diffref)}$ stay bounded as $T \to +\infty$. Similarly,
    \begin{align*}
        \norm{b^{T}-\sigma}_{L^2(\diffref)} = \norm{b^{T}-R_{T}a^{T}b^{T}}_{L^2(\diffref)} \leq \norm{b^{T}}_{L^{\infty}(\diffref)}\norm{R_{T}a^{T}-1}_{L^2(\diffref)}.
    \end{align*}
    We then split $R_{T}a^{T}-1 = R_{T}(a^{T}-\rho)+(R_{T}\rho - 1)$, which vanishes in $L^2(\diffref)$ by the convergence of $a^{T} \to \rho$ in $L^2(\diffref)$. Note that both $L^2(\diffref)$ convergences have rate $\exp(-\lambda_1 T)$ as $T \to +\infty$.
\end{proof}

We now complete the proof of Theorems \ref{thm:compact-manifold} and \ref{thm:compact-manifold-same}.
\begin{proposition}\label{prop:heat-kernel-log-bdd}
    Let $M$ be a compact manifold. There exists $C > 0$ and $T > 0$ such that for all $t \geq T$ and $x \in M$, the following holds for all $y,z \in M$
    \begin{align*}
        \abs{\log \left(\frac{r_t(x,y)}{r_t(x,z)}\right)} \leq C \exp(-\lambda_1 t) d(y,z). 
    \end{align*}
\end{proposition}
\begin{proof}
    Fix $x \in M$. We will apply Proposition \ref{prop:bgl-8-6-1} with  $f(y) = r_{1}(x,y)$. By symmetry of the transition kernel, $R_{t-1}f(z) = r_{t}(x,z)$.
    Let $T_0$ and $C$ be as in the statement of Proposition \ref{prop:bgl-8-6-1}, and without loss of generality let $T_0 \geq 1$. Then, for all $t \geq T_0+1$ and $x \in M$
    \begin{align*}
        \norm{\nabla_{z} r_{t}(x,z)} \leq C \exp(-\lambda_1 (t-1)) \norm{f}_{L^{\infty}(\diffref)}. 
    \end{align*}
    Recall \eqref{eq:uniform-bdd}. Altogether then, there is a constant $K > 0$ such that for all $t \geq T_0+1$, 
    \begin{align*}
        \norm{\nabla_z \log r_t(x,z)} = \norm{\nabla_z r_t(x,z)}/r_t(x,z) \leq K \exp(-\lambda_1 t).
    \end{align*}
    Now, let $\gamma_s : [0,1] \to M$ be a constant speed geodesic from $y$ to $z$, guaranteed to exist by completeness of $M$, then 
    \begin{align*}
        \abs{\log \left(\frac{r_t(x,y)}{r_t(x,z)}\right)} &\leq \int_0^{1} \norm{\nabla_2 \log r_t(x,\gamma_{s})} \norm{\dot{\gamma}_s} ds \leq K \exp(-\lambda_1 t)d(y,z). 
    \end{align*}
\end{proof}

We can now prove Theorem \ref{thm:compact-manifold}.
\begin{proof}[Proof of Theorem \ref{thm:compact-manifold}]
From the definitions of $dP^{T}/dR^{T}$ and $dQ^{T}/dR^{T}$ in \eqref{eq:product-decomp} and \eqref{defn:diffusion-approx-manifold},
\begin{align*}
    \frac{dP^{T}}{dQ^{T}}(x) &=  \frac{a^{T}(x_0)b^{T}(x_T)}{\rho(x_0)\sigma(x_T)}\frac{r_{T/2}(x_T,x_{T/2})}{r_{T/2}(x_T,\Phi_{T/2}(x_{T/2}))}R_{T/2}\sigma(\Phi_{T/2}(x_{T/2})).
\end{align*}
Since the marginal distribution of $x_0$ and $x_{T}$ are the same under both $P_{T}$ and $Q_{T}$, we compute that
\begin{align*}
    H(P_{T}|Q_{T})+H(Q_{T}|P_{T}) &= \left(\Exp{P_{T}}-\Exp{Q_{T}}\right)\left[\log\left(R_{T/2}\sigma(\Phi_{T/2}(x_{T/2}))\right)+\log\left(\frac{r_{T/2}(x_T,x_{T/2})}{r_{T/2}(x_T, \Phi_{T/2}(x_{T/2}))}\right)\right].
\end{align*}
Recall that the first term on the right hand side is controlled by Proposition \ref{prop:first-half-symm-rel-ent}. 

It remains to control the heat kernel term. Observe from Proposition \ref{prop:heat-kernel-log-bdd}, Jensen's inequality, and Proposition \ref{prop:bgl-9-7-2}
\begin{align*}
    \Exp{Q^T}\left[\abs{\log\left(\frac{r_{T/2}(x_T,x_{T/2})}{r_{T/2}(x_T,\Phi_{T/2}(x_{T/2}))}\right)}\right] &\leq \Exp{\diffref}\left[R_{T/2}\rho(x) \cdot K \exp(-\lambda_1 T/2)d(x,\Phi_{T/2}(x))\right] \\
    &= K\exp(-\lambda_1 T/2) \cdot C\exp(-\lambda_1 T/2). 
\end{align*}
The same bound holds for the expression but now under $\mu_{T/2}^{T}$. To see this, 
\begin{align*}
    \Exp{P^T}\left[\abs{\log\left(\frac{r_{T/2}(x_T,x_{T/2})}{r_{T/2}(x_T,\Phi_{T/2}(x_{T/2}))}\right)}\right] &\leq \Exp{\diffref}\left[\frac{d\mu_{T/2}^{T}}{d\diffref} \cdot K \exp(-\lambda_1 T/2)d(x,\Phi_{T/2}(x))\right] \\
    &\leq  K \exp(-\lambda_1 T/2) \norm{\frac{d\mu_{T/2}^{T}}{d\diffref}}_{L^{\infty}(\diffref)}\Exp{\diffref}\left[d(x,\Phi_{T/2}(x))\right].
\end{align*}
Recall that $\norm{d\mu_{T/2}^{T}/d\diffref}_{L^{\infty}(\diffref)}$ stays bounded as $T \to +\infty$ from \eqref{eq:uniform-bdd-potents}. Lastly then, split 
\begin{align*}
    \Exp{\diffref}\left[d(x,\Phi_{T/2}(x))\right] &= \Exp{R_{T/2}^*\mu}[d(x,\Phi_{T/2}(x))]+ \Exp{\diffref}\left[(1-R_{T/2}\rho)d(x,\Phi_{T/2}(x))\right].
\end{align*}
By Jensen's inequality and Proposition \ref{prop:bgl-9-7-2},
\begin{align*}
    \Exp{R_{T/2}^*\mu}[d(x,\Phi_{T/2}(x))] \leq \Was{2}(R_{T/2}^*\mu,R_{T/2}^*\nu)\leq C\exp(-\lambda_1 T/2).
\end{align*}
For the other remaining term, apply Cauchy Schwarz and the spectral decay of $(R_{T/2}\rho - 1)$ to complete the argument (remember that $M$ is bounded). 
\end{proof}

We now establish the tightness of this asymptotic by focusing on the same marginal case.
\begin{proof}[Proof of Theorem \ref{thm:compact-manifold-same}]
    Arguing as in the previous proof, thanks to the fact that $\mu = \nu$ it holds that $\Phi_{T/2} = \Id$ and we just have to consider the log density term. 
    
    \textbf{N.B.} In this proof, we make a different choice for the product decomposition of the $\Schro$ bridge. In  particular, as $\mu = \nu$, we make the symmetric choice of potentials. That is, for each $T > 0$ let $a^{T}: M \to [0,+\infty)$ such that
    \begin{align*}
        \frac{dP^{T}}{dR^{T}} &= a^{T}(x_0)a^{T}(x_{T}). 
    \end{align*}
    Now, the symmetric relative entropy  $H(Q^{T}|P^{T})+H(P^{T}|Q^{T})$ is equal to
\begin{align*}
     \left(\Exp{P_{T}}-\Exp{Q_{T}}\right)\left[\log\left(R_{T/2}\rho\right)(x_{T/2})\right] &= \Exp{\diffref}\left[\left((R_{T/2}a^{T}))^2 - R_{T/2}\rho\right) \log R_{T/2}\rho\right]. 
\end{align*}
Suppose that the following two limits in $L^2(\diffref)$ hold
\begin{align}\label{eq:same-marginal-key-limits}
    \lim\limits_{T \to +\infty} \exp(\lambda_{1}T/2)((R_{T/2}a^{T})^2-R_{T/2}\rho)= \lim\limits_{T \to +\infty} \exp(\lambda_{1}T/2)\log R_{T/2}\rho =  \sum_{i \in \mathcal{I}_1} p_i \psi_i.
\end{align}
As the $(\psi_{i}, i \geq 0)$ form an orthonormal basis of $L^2(\diffref)$, the desired claim follows. We now finish the argument by justifying the limits in \eqref{eq:same-marginal-key-limits}.

We now establish the first limit in \eqref{eq:same-marginal-key-limits}. Let $a^{T}_0 := a^{T} - \int_{M} a^{T} d\diffref$ and $\rho_0 := \rho - 1$ denote the centered versions of $a^T$ and $\rho$, respectively. Next, write the spectral expansion of $a^{T}$ as
\begin{align}
    a^{T} &= \left(\int_{M}a^{T}d\diffref \right)+\sum_{i=1}^{\infty} \alpha_{i}(T) \psi_{i}. 
\end{align}
On the basis of the spectral expansion of $a^T$ and $\rho$, as well as the fact that $a^{T} \to \rho$ in $L^2(\diffref)$ (this is still true despite the different convention from Proposition \ref{prop:schro-limit-facts})
\begin{align}\label{eq:easy-spectral-limits}
    \lim\limits_{T \to +\infty} \exp(\lambda_1 T/2)R_{T/2}\rho_0 = \sum_{i \in \mathcal{I}_1}p_i \psi_{i}, \quad \lim\limits_{T \to +\infty} \exp(\lambda_1 T/2)R_{T/2}a_{0}^{T} = \sum_{i \in \mathcal{I}_1} p_i\psi_{i}.
\end{align}
As $\int_{M} (R_{T/2}a^{T})^2 d\diffref = \int_{M} d\mu_{T/2}^{T} = 1$, the spectral expansion of $a^{T}$ gives that
\begin{align}
    1 &= \norm{R_{T/2}a^{T}}_{L^2(\diffref)}^{2}= \left(\int_{M} a^{T} d\diffref \right)^2+\sum_{i=1}^{\infty} \exp(-\lambda_i T) \alpha_{i}(T)^2. 
\end{align}
This establishes that
\begin{align}\label{eq:spec-lim1}
     \lim\limits_{T \to +\infty} \exp(\lambda_{1}T)\left(\left(\int_{M} a^{T} d\diffref \right)^2 -1 \right) &= -\sum_{i \in \mathcal{I}_1}p_i^2. 
\end{align}
Next, as $\int_{M} a^{T} d\diffref > 0$ a priori, by \eqref{eq:easy-spectral-limits} it holds in $L^2(\diffref)$ that
\begin{align}\label{eq:spec-lim2}
    \lim\limits_{T \to +\infty} \exp(\lambda_{1}T/2) \left(\int_{M}a^{T}d\diffref\right)R_{T/2}a^{T}_0 = 1 \cdot \sum_{i \in \mathcal{I}_1} p_i\psi_i = \sum_{i \in \mathcal{I}_1} p_i\psi_i. 
\end{align}
Lastly, we show in $L^2(\diffref)$ that
\begin{align}\label{eq:hyper-spectral-limit}
    \lim\limits_{T \to +\infty} \exp(\lambda_1 T/2)\norm{(R_{T/2}a_0^{T})^2}_{L^2(\diffref)} = 0.
\end{align}
To see this, we first show that for any $T > 0$, there exists $C_T > 0$ such that the following $L^{2}(\diffref) \to L^4(\diffref)$ bound holds
\begin{align}\label{eq:l2-l4-bdd}
    \norm{R_{T}f}_{L^4(\diffref)} \leq C_{T}\norm{f}_{L^2(\diffref)}.
\end{align}
This follows from the fact that
\begin{align*}
    \abs{R_T f(x)} \leq \left(\int_{M} r_{T}^2(x,y) \diffref(dy)\right)^{1/2} \cdot \norm{f}_{L^2(\diffref)} = \left(\sup \limits_{z} r_{2T}(z,z)\right) ^{1/2}\norm{f}_{L^2(\diffref)}. 
\end{align*}
Hence, it holds that
\begin{align*}
    \norm{R_T f}^{4}_{L^{4}(\diffref)} \leq \norm{R_T f}^2_{L^{\infty}(\diffref)} \norm{R_T f}^2_{L^{2}(\diffref)} \leq \left(\sup \limits_{z} r_{2T}(z,z)\right)\norm{f}_{L^2(\diffref)}^{4}. 
\end{align*}
This establishes \eqref{eq:l2-l4-bdd} with $C_T = \left(\sup \limits_{z} r_{2T}(z,z)\right)^{1/4}$. Fix some $T_0 > 0$, then  whenever $T/2 > T_0$
\begin{align*}
    \norm{(R_{T/2}a_0^{T})^2}_{L^2(\diffref)} &= \norm{R_{T/2}a_0^{T}}_{L^4(\diffref)}^2 \leq C_{T_0}^2\norm{R_{(T/2)-T_0}a_0^{T}}_{L^2(\diffref)}^2 \leq C_{T_0}^2\exp(-\lambda_1 (T-2T_0))\norm{a_0^{T}}_{L^2(\diffref)}^2.
\end{align*}

As $(a^{T}, T> 0)$ has an $L^2(\diffref)$ limit, this establishes \eqref{eq:hyper-spectral-limit}. Finally then, by \eqref{eq:easy-spectral-limits}, \eqref{eq:spec-lim1}, \eqref{eq:spec-lim2}, and \eqref{eq:hyper-spectral-limit}, we see in $L^2(\diffref)$ that
\begin{align*}
    &\lim\limits_{T \to +\infty}\exp(\lambda_{1}T/2)\left(((R_{T/2}a^{T})^2-R_{T/2}\rho)\right)\\
    &= \lim\limits_{T \to +\infty}\exp(\lambda_{1}T/2)\left((R_{T/2}a^{T}_0)^2 + 2\left(\int_M a^{T}d\diffref \right)R_{T/2}a^{T}_0+\left(\int_M a^{T}d\diffref \right)^2-1 - R_{T/2}\rho_0\right)\\
    &= \sum_{i \in \mathcal{I}_1} p_{i}\psi_{i},
\end{align*}
establishing the first limit in \eqref{eq:same-marginal-key-limits}.

Next, we establish the second limit in \eqref{eq:same-marginal-key-limits}. The following pointwise inequality holds
\begin{align}
    \abs{\log R_{T/2}\rho - \left(R_{T/2}\rho - 1\right)} \leq \frac{1}{R_{T/2}\rho}\left(R_{T/2}\rho - 1\right)^2. 
\end{align}
By \eqref{eq:uniform-bdd}, for all $T$ large enough it holds that $R_{T/2}\rho > c$ for some $c > 0$. Repeating the same argument for \eqref{eq:hyper-spectral-limit} establishes that
\begin{align}\label{eqref:hyper-contrac-density}
    \lim\limits_{T \to +\infty} \exp(\lambda_1 T/2)\norm{(R_{T/2}\rho - 1)^2}_{L^2(\diffref)} = 0.
\end{align}
Hence, thanks to the first limit in \eqref{eq:easy-spectral-limits} the second limit in \eqref{eq:same-marginal-key-limits} holds. 
\end{proof}

\subsection{Euclidean Proofs}
First, we establish the Gaussian calculation referenced after the statement of Theorem \ref{thm:euclidean-case}.
\begin{proposition}\label{prop:euc-gaussian-comp}
    In the setting of Section \ref{sec:ou-diff}, let $\mu = N(0,1)$ and $\nu = N(0,\eta^2)$. Then $P^{T}$ and $Q^{T}$ can be explicitly computed, and setting $\pi^{T} = (x_0,x_{T})_{\#}P^{T}$ and $q^{T} = (x_0,x_{T})_{\#}Q^{T}$ we compute
    \begin{align*}
        \lim\limits_{T \to +\infty }\exp(2T)\left(H(\pi^{T}|q^{T})+H(q^{T}|\pi^{T})\right) < +\infty.
    \end{align*}
    Hence, the rate for the endpoint distributions suggested by Theorem \ref{thm:euclidean-case} is not in general tight.
\end{proposition}

\begin{proof}
    In this proof, we will justify the claim that $\pi^{T} \sim N(0,\Sigma_{1})$ and $q_{T} \sim N(0,\Sigma_2)$ with
    \begin{align}
        \Sigma_{1} = \begin{pmatrix} 1 & c_{T} \\ c_{T} & \eta^2 \end{pmatrix}, \quad \Sigma_2 = \begin{pmatrix} 1 & \frac{e^{-T/2}\eta^2}{\sqrt{1+e^{-T/2}(\eta^2-1)}} \\ \frac{e^{-T/2}\eta^2}{\sqrt{1+e^{-T/2}(\eta^2-1)}} & \eta^2 \end{pmatrix},
    \end{align}
    where $c_{T} = \frac{1}{2e^{-T/2}}\left(\sqrt{(1-e^{-T})^2+4e^{-T}\eta^2}-(1-e^{-T})\right)$.
    
    To compute the covariance matrix of $\pi^{T}$, we follow the calculation developed in \cite[Appendix A.4]{zhang-sb-ouformula-2026}. We note that similar explicit $\Schro$ bridge calculations are provided in \cite{bunne-sbformula-26,janati2020}. Let $(X_0,X_{T}) \sim \pi^{T}$, then the key quantity to compute is $\Cov(X_0,X_T)$. From the change of variables outlined in \cite[Appendix A.4, eqn. (70-72)]{zhang-sb-ouformula-2026}, we define $\bar{X}_0 = (e^{-T/2}/\sqrt{1-e^{-T}})X_0$ and $\bar{X}_T = (1/\sqrt{1-e^{-T}})X_T$. In their notation, we set $\mathcal{A} = 1$, $\mathcal{B} = \eta^2$, $a = b = 0$, $m = 0$, $A = \frac{1}{2}$, $\sigma = 1$, and $\Sigma_{T} = 1-e^{-T}$. Altogether, \cite[eqn. (64)]{zhang-sb-ouformula-2026} gives that
    \begin{align*}
        \Cov(\bar{X}_0,\bar{X}_T) &= \frac{1}{2(1-e^{-T})}\left(\sqrt{4\eta^2 e^{-T} + (1-e^{-T})^2} - (1-e^{-T})\right).
    \end{align*}
    By the bilinearity of covariance, it follows that $c_{T} = \Cov(X_0,X_T)$. 
    
    To compute the covariance matrix of $q^{T}$, we simply compute the covariance of the triplet $(X_0,X_{T/2},X_{T})$ defined in \eqref{eq:triplet-ou}. Let $(Y_t, t \geq 0)$ denote a solution to \eqref{eq:OU} with $Y_0 \sim N(0,\eta^2)$, then it holds that
    \begin{align*}
        \begin{pmatrix}
            X_0 \\ X_{T/2}
        \end{pmatrix} \sim 
        N\left(\begin{pmatrix} 0 \\ 0 \end{pmatrix},\begin{pmatrix} 1 & e^{-T/4} \\ e^{-T/4} & 1 \end{pmatrix} \right), 
        \begin{pmatrix}
            Y_0 \\ Y_{T/2}
        \end{pmatrix} \sim 
        N\left(\begin{pmatrix} 0 \\ 0 \end{pmatrix},\begin{pmatrix} \eta^2 & \eta^2 e^{-T/4} \\ \eta^2 e^{-T/4} & 1 + (\eta^2 -1)e^{-T/2} \end{pmatrix} \right).
    \end{align*}
    As $X_{T/2} \sim N(0,1)$ and $Y_{T/2} \sim N(0,1+(\eta^2-1)e^{-T/2})$, it follows that $\nabla \varphi_{T/2}(x) := \sqrt{1+(\eta^2-1)e^{-T/2}}x$. Next, by construction the conditional distribution $X_{T}|(X_{T/2},X_0)$ is equal to 
    \begin{align*}
        Y_{0}|(Y_{T/2} = \nabla \varphi_{T/2}(X_{T/2})) \sim N\left(\frac{\eta^2 e^{-T/4}}{\sqrt{1+(\eta^2-1)e^{-T/2}}}X_{T/2}, \frac{\eta^2(1- e^{-T/2})}{1+(\eta^2-1)e^{-T/2}}\right)
    \end{align*}
    Altogether, the triplet $(X_0,X_{T/2},X_{T})$ is a centered Gaussian vector with covariance matrix
    \begin{align*}
        \begin{pmatrix}
            1 & e^{-T/4} & \frac{e^{-T/2}\eta^2}{\sqrt{1+e^{-T/2}(\eta^2-1)}} \\
            e^{-T/4} & 1 & \frac{e^{-T/4}\eta^2}{\sqrt{1+e^{-T/2}(\eta^2 - 1)}} \\ \frac{e^{-T/2}\eta^2}{\sqrt{1+e^{-T/2}(\eta^2-1)}} & \frac{e^{-T/4}\eta^2}{\sqrt{1+e^{-T/2}(\eta^2 - 1)}} & \eta^2
        \end{pmatrix}.
    \end{align*}

We finally proceed to the relative entropy calculation. With $c_{T}$ defined above, define in addition $d_T =  \frac{e^{-T/2}\eta^2}{\sqrt{1+e^{-T/2}(\eta^2-1)}}$. Using the formula for the relative entropy between Gaussians, we compute for all $T > 0$ that
    \begin{align*}
        H(\pi^{T}|q^{T})+H(q^{T}|\pi^{T}) &= \frac{1}{2}\Tr(\Sigma_{1}^{-1}\Sigma_{2})+\frac{1}{2}\Tr(\Sigma_{2}^{-1}\Sigma_{1})-2\\
        &= \frac{(c_T-d_T)^2(\eta^2 + c_T d_T)}{(\eta^2-c_{T}^2)(\eta^2-d_{T}^2)}.
    \end{align*}
    It now remains to justify that the above quantity is $O(\exp(-2T))$. First, observe that the Taylor expansion of $\sqrt{(1-x)^2+4\eta^2 x}$ about $x = 0$ is given by $1 + (2\eta^2 -1)x+O(x^2)$. Hence, for large $T$ it holds that
    \begin{align}\label{eq:ct-asym}
        c_{T} &= \frac{1}{2e^{-T/2}}\left(1+(2\eta^2 - 1)e^{-T}+O(e^{-2T})-(1-e^{-T})\right) = \eta^2 e^{-T/2}+O(e^{-3T/2}). 
    \end{align}
    Similarly, from the expansion $(1+x)^{-1/2} = 1 - \frac{1}{2}x+ O(x^2)$ for $\abs{x} < 1$, we obtain that
    \begin{align}\label{eq:dt-asymp}
        d_{T} = \eta^{2}e^{-T/2}\left(1-\frac{1}{2}e^{-T/2}(\eta^2 - 1) + O(e^{-T})\right).
    \end{align}
    From \eqref{eq:ct-asym} and \eqref{eq:dt-asymp}, we deduce that the decay in the symmetric relative entropy is governed by $(c_T-d_T)^2$, and we further see that this quantity is $O(\exp(-2T))$.
\end{proof}

Next, we prove Proposition \ref{prop:prop-of-log-conc}.
\begin{proof}[Proof of Proposition \ref{prop:prop-of-log-conc}]
    Let $p$ be the Lebesgue density of $\mu$, then $p$ is bounded and compactly supported. Let $X \sim \mu$ and $Z \sim N(0,\Id)$ be independent, then 
    \begin{align*}
        e^{-t/2}\left(X+\sqrt{e^{t}-1}Z\right) \sim R_{t}^*\mu.
    \end{align*}
    Let $(p_t(\cdot,\cdot), t > 0)$ denote the standard Brownian transition densities, i.e. $p_t(x,\cdot)$ is the Lebesgue density of the $N(x,t\Id)$ random variable. Let $(P_t, t \geq 0)$ denote the corresponding semigroup. The Lebesgue density of $R_{t}^*\mu$ is equal to $e^{dt/2}P_{e^{t}-1}p(e^{t/2}x)$. Let $R> 0$ be such that $\mathrm{supp}(p)\subset B(0,R)$. By the calculations in \cite[Example 2.3]{brigati-heatflow25}, see also \cite[Section 2.1]{gozlan-fi-compact18}, for each $t > 0$
    \begin{align*}
        \frac{1}{t}\left(1-\frac{R^2}{t}\right)\Id \leq -\nabla^2 \log P_{t}p \leq \frac{1}{t}\Id. 
    \end{align*} 
    With the \eqref{eq:OU} rescaling, it holds that
    \begin{align*}
        \frac{e^{t}}{e^{t}-1}\left(1-\frac{R^2}{e^{t}-1}\right) \leq -\nabla^2 \log \left(\frac{dR_{t}^*\mu}{d\mathrm{Leb}}\right) \leq \frac{e^{t}}{e^{t}-1}\Id.
    \end{align*}
    As such, there exist $0 < \alpha < \beta$ and $T_0$ such that for all $T \geq T_0$, 
    \begin{align*}
        \alpha \Id \leq -\nabla^2 \log \left(\frac{dR_{T/2}^*\mu}{d\mathrm{Leb}}\right),-\nabla^2 \log \left(\frac{dR_{T/2}^*\nu}{d\mathrm{Leb}}\right) \leq \beta \Id.
    \end{align*}
    By the Caffarelli Contraction Theorem \cite{THEcaffarelli}, for example as stated in \cite[Theorem 1]{chewi2022entropic}, the proposition follows with $C = \sqrt{\beta/\alpha}$. 

    For the global pointwise bound, by the Mean Value Theorem it holds for all $T \geq T_0$ that
    \begin{align*}
        \norm{\nabla \varphi_{T/2}(x)} \leq \norm{\nabla \varphi_{T/2}(0)}+C\norm{x}.
    \end{align*}
    Thus, the global pointwise bound holds once we establish that $\norm{\nabla \varphi_{T/2}(0)}$ is bounded for all $T$ large enough. To see this, observe again from the Mean Value Theorem
    \begin{align*}
        \norm{\nabla \varphi_{T/2}(0)} \leq \norm{\nabla \varphi_{T/2}(x)} + C \norm{x}, \quad \text{for all $T \geq T_0$ and $x \in \mathbb{R}^{d}$}.  
    \end{align*}
    Integrating the inequality with respect to $R_{T/2}^*\mu$ to get that
    \begin{align*}
        \norm{\nabla \varphi_{T/2}(0)} &\leq \int_{\mathbb{R}^{d}} \norm{y} R_{T/2}^*\nu(dy) + C \int_{\mathbb{R}^{d}} \norm{x} R_{T/2}^*\mu(dx). 
    \end{align*}
    By Assumption \ref{assumption}, $dR_{T/2}^*\mu/d\diffref = R_{T/2}\rho$ and $dR_{T/2}^*\nu/d\diffref=R_{T/2}\sigma$ are uniformly bounded as $T \to +\infty$. Thus, as $\diffref$ has finite moments, this gives the desired pointwise bound. 
\end{proof}

Next, we prove the following heat kernel bound in the \eqref{eq:OU} setting, similar to Proposition \ref{prop:heat-kernel-log-bdd}.
\begin{proposition}\label{prop:ou-density-ratio}
    Let $(r_{t}(\cdot,\cdot), t \geq 0)$ denote the transition kernel of the OU process defined in \eqref{eq:ou-trans-dens}. Under Assumption \ref{assumption},
    \begin{align*}
        \limsup\limits_{T \to +\infty} \exp(T/2)\abs{\left(\Exp{P_T}-\Exp{Q_{T}}\right) \left[\log\left(\frac{r_{T/2}(x_T,x_{T/2})}{r_{T/2}(x_T,\nabla \varphi_{T/2}(x_{T/2}))}\right)\right]} < +\infty. 
    \end{align*}
\end{proposition}
\begin{proof} 
    Recall that for \eqref{eq:OU}, $\lambda_1 = \frac{1}{2}$. Computing directly from \eqref{eq:ou-trans-dens},
    \begin{align}\label{eq:log-ou-ratio}
        \log\left(\frac{r_{T/2}(x,y)}{r_{T/2}(x,z)}\right) &= \frac{e^{-T/4}}{2(1-e^{-T/2})}\langle 2x-e^{-T/4}(y+z),y-z\rangle.
    \end{align}
    By Proposition \ref{prop:prop-of-log-conc}, there is $C > 0$ such that $\norm{\nabla \varphi_{T/2}(x)} \leq C(1+\norm{x})$ for all $T$ large enough. It then holds for all $T$ large enough that
    \begin{align*}
        \Exp{Q^{T}}\left[\norm{2x_{T}-e^{-T/4}(x_{T/2}+\nabla \varphi_{T/2}(x_{T/2}))}^2\right] \leq 4 \left(\int \norm{x}^2 d\nu + 4 \exp(-T/2)\left(\int \norm{x}^2+C(1+\norm{x})^2 \right)dR_{T/2}^*\mu\right),
    \end{align*}
    which stays bounded as $T \to +\infty$. By Cauchy-Schwarz and Proposition \ref{prop:bgl-9-7-2}, there exists $K > 0$ such that for all $T$ large enough
    \begin{align*}
        \abs{\Exp{Q_{T}}\left[\log\left(\frac{r_{T/2}(x_T,x_{T/2})}{r_{T/2}(x_T,\nabla \varphi_{T/2}(x_{T/2}))}\right)\right]} \leq K\exp(-T/4)\Was{2}(R_{T/2}^*\mu,R_{T/2}^*\nu) \leq K\exp(-T/2)\Was{2}(\mu,\nu).
    \end{align*}

    For the $P^{T}$ term, observe that under $P^T$, $(x_{T/2},x_{T})$ is distributed according to
    \begin{align*}
        R_{T/2}a^{T}(x_{T/2})b^{T}(x_{T})r_{T/2}(x_{T/2},x_{T})(\diffref \otimes \diffref)(dx_{T/2}dx_{T}).
    \end{align*}
    By \eqref{eq:uniform-bdd-potents}, there exists $T_0, C > 0$ such that for all $T \geq T_0$, the joint density of $(x_{T/2},x_{T})$ under $P^T$ is dominated by $C r_{T/2}(x_{T/2},x_{T})(\diffref \otimes \diffref)(dx_{T/2}dx_{T})$. 
    Next, observe from \eqref{eq:log-ou-ratio} that
    \begin{align*}
        \abs{\log\left(\frac{r_{T/2}(x,y)}{r_{T/2}(x,z)}\right)} \leq \exp(-T/4)\norm{x}\norm{y-z} + \exp(-T/2)\norm{y+z}\norm{y-z}
    \end{align*}    
    Combining these facts gives that
    \begin{align*}
        &\abs{\Exp{P_{T}}\left[\log\left(\frac{r_{T/2}(x_T,x_{T/2})}{r_{T/2}(x_T,\nabla \varphi_{T/2}(x_{T/2}))}\right)\right]} \\
        &\leq C\exp(-T/4)\int_{\mathbb{R}^{d} \times \mathbb{R}^{d}} \norm{x_T}\norm{x_{T/2}-\nabla \varphi_{T/2}(x_{T/2})}r_{T/2}(x_T,x_{T/2})(\diffref \otimes \diffref)(dx_{T/2}dx_{T})\\ &+C \exp(-T/2)\int_{\mathbb{R}^{d} } \norm{x+\nabla \varphi_{T/2}(x)}\norm{x-\nabla \varphi_{T/2}(x)}d\diffref(x).
    \end{align*}
    Thanks to the uniform in $T$ linear-growth bounds of Proposition \ref{prop:prop-of-log-conc} and the fact that $\diffref$ has finite moments of all orders, last term in the above display decays like $\exp(-T/2)$. Thus, it remains to show that the first integral in the above upper bound decays like $\exp(-T/4)$. Recall from the spectral expansion of $r_{T/2}(\cdot,\cdot)$ in $L^2(\diffref \otimes \diffref)$ that there is a $C > 0$ such that for all $T > 0$ large enough 
    \begin{align*}
        \norm{r_{T/2}(x,y) - 1}_{L^2(\diffref \otimes \diffref)} \leq C\exp(-T/4). 
    \end{align*}
    This then implies that
    \begin{align*}
        &\abs{\int_{\mathbb{R}^{d} \times \mathbb{R}^{d}} \norm{x_T}\norm{x_{T/2}-\nabla \varphi_{T/2}(x_{T/2})}(r_{T/2}(x_T,x_{T/2})-1)(\diffref \otimes \diffref)(dx_{T/2}dx_{T})}\\
        &\leq \left(\int_{\mathbb{R}^{d} \times \mathbb{R}^{d}} \norm{x_T}^2 \norm{x_{T/2}-\nabla \varphi_{T/2}(x_{T/2})}^{2}(\diffref \otimes \diffref)(dx_{T/2}dx_{T})\right)^{1/2}\norm{r_{T/2}(x_T,x_{T/2})-1}_{L^2(\diffref \otimes \diffref)}.
    \end{align*}
    The integral stays bounded as $T \to +\infty$, and the heat kernel term provides the $\exp(-T/4)$ decay. Altogether then, the argument is completed once we establish the existence of some $C > 0$ such that for all $T$ large enough,
    \begin{align}\label{eq:last-claim}
        \int_{\mathbb{R}^{d} \times \mathbb{R}^{d}} \norm{x_T}\norm{x_{T/2}-\nabla \varphi_{T/2}(x_{T/2})}(\diffref \otimes \diffref)(dx_{T/2}dx_{T}) \leq C \exp(-T/4).
    \end{align}
    But this follows from standard spectral arguments. Indeed, note that since $x_{T/2}$ and $x_{T}$ are independent under $\diffref \otimes \diffref$ and $\diffref$ has finite moments, it remains to analyze the convergence rate of
    \begin{align}\label{eq:last-thing-to-bdd}
        \int_{\mathbb{R}^{d}} \norm{x-\nabla \varphi_{T/2}(x)}\diffref(dx).
    \end{align}
    But now we apply the same maneuver as in the proof of Theorem \ref{thm:compact-manifold}. Namely, we bound \eqref{eq:last-thing-to-bdd} with
    \begin{align*}
        \int_{\mathbb{R}^{d}} \norm{x-\nabla \varphi_{T/2}(x)}\cdot \abs{R_{T/2}\rho(x) - 1}\diffref(dx) + \Exp{R_{T/2}\mu^*}\left[\norm{x-\nabla \varphi_{T/2}(x)}\right]
    \end{align*}
    By Jensen's inequality and Proposition \ref{prop:bgl-9-7-2}, the second term decays with rate $\exp(-T/4)$. For the first term, apply Cauchy Schwarz and use the spectral decay of $R_{T/2}\rho$ to get a rate of $\exp(-T/4)$. This establishes \eqref{eq:last-claim}, completing the argument. 
\end{proof}

Finally then Theorem \ref{thm:euclidean-case} is established in the following manner.
\begin{proof}[Proof of Theorem 2]
    Just as in the proof of Theorem \ref{thm:compact-manifold}, the sum of the two relative entropy quantities is the sum of two terms. The first term is  controlled by Proposition \ref{prop:first-half-symm-rel-ent}, and the remaining heat kernel term is bounded by Proposition \ref{prop:ou-density-ratio}.
\end{proof}

\bibliographystyle{alpha}
\bibliography{sample}

\end{document}